\documentclass[sn-basic]{sn-jnl}% Basic Springer Nature Reference Style/Chemistry Reference Style

\usepackage{graphicx}%
\usepackage{multirow}%
\usepackage{amsmath,amssymb,amsfonts}%
\usepackage{amsthm}%
\usepackage{mathrsfs}%
\usepackage[title]{appendix}%
\usepackage{xcolor}%
\usepackage{textcomp}%
\usepackage{manyfoot}%
\usepackage{booktabs}%
\usepackage{algorithm}%
\usepackage{algorithmicx}%
\usepackage{algpseudocode}%
\usepackage{listings}%
\usepackage[capitalize]{cleveref}
\usepackage{braket}
\usepackage{nicefrac}

\usepackage{stackengine}
\usepackage{braket}
\usepackage{amsmath}
\usepackage{amscd}
\usepackage{amsfonts}
\usepackage{amssymb}
\usepackage{subfigure}
\usepackage{tabularray}
\UseTblrLibrary{booktabs}
\usepackage{hyperref} % Bei cleverref poorman auskommentieren
\usepackage{pgfplots}
\pgfplotsset{compat=newest}
\usetikzlibrary{calc}
\usepackage{tikzscale}
\usepackage{tikz-3dplot}
\usepackage[capitalize]{cleveref}
\crefname{equation}{}{}
\usepackage{nicefrac}

\theoremstyle{thmstyleone}%
\newtheorem{theorem}{Theorem}%  meant for continuous numbers
\newtheorem{proposition}[theorem]{Proposition}% 
\newtheorem{corollary}[theorem]{Corollary}% 

\theoremstyle{thmstyletwo}%
\newtheorem{lemma}{Lemma}%
\newtheorem{remark}{Remark}%

\theoremstyle{thmstylethree}%

\usepackage[varvw,slantedGreek]{newtxmath}
\newcommand{\N}{\varmathbb{N}}

\newcommand{\R}{\varmathbb{R}}

\newcommand{\OO}{{\mathcal O}}

\newcommand{\T}{\mathcal{T}}
\newcommand{\abs}[1]{\lvert#1\rvert}

\newcommand{\norm}[1]{\lVert#1\rVert}

\newcommand{\ltwonorm}[1]{\norm{#1}_{L^2(\Omega)}}
\newcommand{\ltwonormdo}[1]{\norm{#1}_{L^2(\partial \Omega)}}

\newcommand{\honenorm}[1]{\norm{#1}_{H^1(\Omega)}}
\newcommand{\htwonorm}[1]{\norm{#1}_{H^2(\Omega)}}

\newcommand{\into}{\int_\Omega}

\newcommand{\half}{\frac{1}{2}}
\newcommand{\nhalf}{\nicefrac{1}{2}}

\newcommand{\dist}{\operatorname{dist}}

\renewcommand{\Re}{\operatorname{Re}}

\newcommand{\oq}{\bar q}
\newcommand{\ou}{\bar u}
\newcommand{\oz}{\bar z}

\newcommand{\dq}{\delta\mspace{-2mu}q}
\newcommand{\du}{\delta\mspace{-2mu}u}

\newcommand{\lh}{\abs{\ln h}}
\newcommand{\Lap}{\upDelta}
\renewcommand{\phi}{\varphi}
\newcommand{\Om}{\Omega}

\newcommand{\logLogSlopeTriangle}[6]
{
	\pgfplotsextra
	{
		\pgfkeysgetvalue{/pgfplots/xmin}{\xmin}
		\pgfkeysgetvalue{/pgfplots/xmax}{\xmax}
		\pgfkeysgetvalue{/pgfplots/ymin}{\ymin}
		\pgfkeysgetvalue{/pgfplots/ymax}{\ymax}
		
		\pgfmathsetmacro{\xArel}{#1}
		\pgfmathsetmacro{\yArel}{#3}
		\pgfmathsetmacro{\xBrel}{#1-#2}
		\pgfmathsetmacro{\yBrel}{\yArel}
		\pgfmathsetmacro{\xCrel}{\xArel}
		
		\pgfmathsetmacro{\lnxB}{\xmin*(1-(#1-#2))+\xmax*(#1-#2)} % in [xmin,xmax].
		\pgfmathsetmacro{\lnxA}{\xmin*(1-#1)+\xmax*#1} % in [xmin,xmax].
		\pgfmathsetmacro{\lnyA}{\ymin*(1-#3)+\ymax*#3} % in [ymin,ymax].
		\pgfmathsetmacro{\lnyC}{\lnyA+#4*(\lnxA-\lnxB)}
		\pgfmathsetmacro{\yCrel}{(\lnyC-\ymin)/(\ymax-\ymin)} % THE IMPROVED EXPRESSION WITHOUT 'DIMENSION TOO LARGE' ERROR.
		
		\coordinate (A) at (rel axis cs:\xArel,\yArel);
		\coordinate (B) at (rel axis cs:\xBrel,\yBrel);
		\coordinate (C) at (rel axis cs:\xCrel,\yCrel);
		
		\draw[#5]   (A)-- %node[pos=0.5,anchor=north] {1}
		(B)-- 
		(C)-- node[pos=0.5,anchor=west] {$#6$}
		cycle;
	}
}

\newcommand{\logLogSlopeTriangleInverse}[6]
{
	\pgfplotsextra
	{
		\pgfkeysgetvalue{/pgfplots/xmin}{\xmin}
		\pgfkeysgetvalue{/pgfplots/xmax}{\xmax}
		\pgfkeysgetvalue{/pgfplots/ymin}{\ymin}
		\pgfkeysgetvalue{/pgfplots/ymax}{\ymax}
		
		\pgfmathsetmacro{\xArel}{#1}
		\pgfmathsetmacro{\yArel}{#3}
		\pgfmathsetmacro{\xBrel}{#1-#2}
		\pgfmathsetmacro{\yBrel}{\yArel}
		\pgfmathsetmacro{\xCrel}{\xBrel}
		
		\pgfmathsetmacro{\lnxB}{\xmin*(1-(#1-#2))+\xmax*(#1-#2)} % in [xmin,xmax].
		\pgfmathsetmacro{\lnxA}{\xmin*(1-#1)+\xmax*#1} % in [xmin,xmax].
		\pgfmathsetmacro{\lnyA}{\ymin*(1-#3)+\ymax*#3} % in [ymin,ymax].
		\pgfmathsetmacro{\lnyC}{\lnyA+#4*(\lnxA-\lnxB)}
		\pgfmathsetmacro{\yCrel}{2*\yArel-(\lnyC-\ymin)/(\ymax-\ymin)} % THE IMPROVED EXPRESSION WITHOUT 'DIMENSION TOO LARGE' ERROR.
		
		\coordinate (A) at (rel axis cs:\xArel,\yArel);
		\coordinate (B) at (rel axis cs:\xBrel,\yBrel);
		\coordinate (C) at (rel axis cs:\xCrel,\yCrel);
		
		\draw[#5]   (A)-- %node[pos=0.5,anchor=north] {1}
		(B)-- node[pos=0.5,left,anchor=east] {$#6$}
		(C)-- 
		cycle;
	}
}

\begin{document}

\title[Neumann Optimal Control on Convex Domains]{Numerical Analysis for Neumann Optimal Control Problems on Convex Polyhedral Domains}


\author*[1]{\fnm{Johannes} \sur{Pfefferer}}\email{johannes.pfefferer@unibw.de}

\author[2]{\fnm{Boris} \sur{Vexler}}\email{vexler@tum.de}

\affil[1]{\orgdiv{Fakult\"at f\"ur Elektrische Energiesysteme und Informationstechnik}, \orgname{Universit\"at der Bundeswehr M\"unchen}, \orgaddress{\street{Werner-Heisenberg-Weg 39}, \city{Neubiberg}, \postcode{85579}, \country{Germany}}}
\affil[2]{\orgdiv{Department of Mathematics}, \orgname{Technical University of Munich, School of Computation, Information and Technology}, \orgaddress{\street{Boltzmannstr. 3}, \city{Garching b. Munich}, \postcode{85748}, \country{Germany}}}


\abstract{%
This paper is concerned with finite element error estimates for Neumann boundary control problems posed on convex and polyhedral domains. Different discretization concepts are considered and for each optimal discretization error estimates are established. In particular, for a full discretization with piecewise linear and globally continuous functions for the control and standard linear finite elements for the state optimal convergence rates for the controls are proven which solely depend on the largest interior edge angle. To be more precise, below the critical edge angle of $2\pi/3$, a convergence rate of two (times a log-factor) can be achieved for the discrete controls in the $L^2$-norm on the boundary. For larger interior edge angles the convergence rates are reduced depending on their size, which is due the impact of singular (domain dependent) terms in the solution. The results are comparable to those for the two dimensional case. However, new techniques in this context are used to prove the estimates on the boundary which also extend to the two dimensional case. Moreover, it is shown that the discrete states converge with a rate of two in the $L^2$-norm in the domain independent of the interior edge angles, i.e. for any convex and polyhedral domain. It is remarkable that this not only holds for a full discretization using piecewise linear and globally continuous functions for the control, but also for a full discretization using piecewise constant functions for the control, where the discrete controls only converge with a rate of one in the $L^2$-norm on the boundary at best.
At the end, the theoretical results are confirmed by several numerical experiments.
}

\keywords{Optimal control, Neumann boundary control, convex polyhedral domains, finite element method, discretization error estimates,}


\pacs[MSC Classification]{35J05, 49J20, 49M25, 65N15, 65M30}

\maketitle
\section{Introduction}\label{sec1}

In this paper we consider the following optimal control problem with control entering the Neumann boundary condition of a linear elliptic equation: 
\begin{subequations}\label{NeumannCon:eq:problem}
	\begin{equation}\label{NeumannCon:eq:cost}
		\text{Minimize } J(q,u) = \half\ltwonorm{u-u_d}^2 + \frac{\alpha}{2}\ltwonormdo{q}^2
	\end{equation}
	over control $q$ and state $u$ fulfilling the state equation
	\begin{equation}\label{NeumannCon:eq:state}
		\begin{aligned}
			-\Lap u + u&= 0 &\quad&\text{in } \Omega,\\
			\partial_n u &=q &\quad&\text{on } \partial \Omega,
		\end{aligned}
	\end{equation}
\end{subequations}
where $\Omega \subset \R^3$ is a convex polyhedral domain. Such problems are often referred as \textit{Neumann boundary control problems}. The precise functional analytic setting is discussed below.
For simplicity of presentation we do not include inequality constraints on the control variable. However, most of the results can be extended to problems with pointwise control constraints of the form
\begin{equation}\label{NeumannCon:eq:constraints}
	q_a \le q(s) \le q_b \quad \text{for almost all } s \in \partial \Omega
\end{equation}
with $q_a, q_b\in \R$. We will remark on such extensions below.

The goal of the paper is to provide precise numerical analysis for the finite element discretization of the above boundary control problem.
In particular, this is based on error estimates for an adjoint equation, which is also a linear elliptic equation with homogeneous Neumann boundary condition and an inhomogeneous right-hand side. Thus, we require certain results for the discretization of the equation
\begin{equation}\label{NeumannCon:eq:into:z}
	\begin{aligned}
		-\Lap z + z&= f &\quad&\text{in } \Omega,\\
		\partial_n z &=0 &\quad&\text{on } \partial \Omega.
	\end{aligned}
\end{equation}
Let $z_h$ be a Galerkin finite element approximation (with linear finite elements) of $z$. For the error analysis of the control problem error estimates on the boundary are required, i.e. estimates for $\ltwonormdo{z-z_h}$. The classical result for this error is
\[
\ltwonormdo{z-z_h} \le c h^{\frac{3}{2}} \htwonorm{z},
\]
which is optimal if we only require the $H^2$ regularity, i.e. $z \in H^2(\Omega)$. Depending on the geometry of the domain and the regularity of the right-hand side $f$ higher regularity of $z$ can be expected. A natural question is, if it is possible to prove
\begin{equation}\label{intro:eq:z_est_1}
	\ltwonormdo{z-z_h} \le c h^{\frac{3}{2}+s} \norm{z}_{H^{2+s}(\Omega)},
\end{equation}
for $0<s<\half$ provided $z \in H^{2+s}(\Omega)$. Although this is a natural extension of the $\OO(h^{\frac{3}{2}})$ result, it does not seem to be available in the literature. We prove an even stronger result, see \cref{FEM:Theorem:u-uhboundary} below, i.e.
\begin{equation}\label{intro:eq:z_est_2}
	\ltwonormdo{z-z_h} \le c h^{\frac{3}{2}+s} \norm{z}_{H^2_{-s}(\Omega)},
\end{equation}
for $0<s<\half$ provided that $z$ belongs to a weighted space $H^2_{-s}(\Omega)$ with the weight being the distance to the boundary. By fractional Hardy inequality, see \cref{theorem:HardyFractional} below, there holds $H^{2+s}(\Omega) \hookrightarrow H^2_{-s}(\Omega)$ and thus the result \eqref{intro:eq:z_est_1} follows. We also provide conditions on $z$ to possess the regularity $H^2_{-s}(\Omega)$. To this end we consider the critical exponent $\lambda_\Omega$ depending only on the edge openings of $\partial \Omega$, see  \cref{eq:lambda_Omega} below. Then for $s < s_\Omega = \min(\lambda_\Omega-1,\half)$ and $f\in H^1(\Omega)$ we prove $z \in H^2_{-s}(\Omega)$ and the corresponding a priori and discretization error estimates. To be more precise, we prove for every $s<s_\Omega$
\[
\ltwonormdo{z-z_h} \le c h^{\frac{3}{2}+s} \norm{f}_{H^1(\Omega)},
\]
see \cref{FEM:Corollary:u-uhboundary1}. Under the additional assumption that all interior edge angles are smaller then $\frac{2 \pi}{3}$ we prove
\[
\ltwonormdo{z-z_h} \le ch^2\lh\norm{f}_{H^1(\Omega)},
\]
see \cref{FEM:Corollary:u-uhboundary2}. Thus, we can achieve a convergence rate of two (times a $\log$-factor). These two results are comparable to those for the two dimensional case proven in \cite{ApelPfeffererRoesch:2015,Pfefferer:Thesis:2014}, except that there the right hand side $f$ needs to be H\"older continuous and the $\log$-factor has a larger exponent. In addition, our techniques, which are new in this context, seem to be more natural than those in the literature and also extend to the two dimensional case. They are based on weighted error estimates with weight being the (regularized) distance function to the boundary, cf. our recent paper on numerical analysis for Dirichlet control problems \cite{PfeffererVexler:2024}. The main tool is a weighted best approximation result, see \cref{FEM:Neumann:weighted_best_approximation}, which is also of an independent interest.

We then use the finite element error estimates on the boundary for the adjoint state, to obtain optimal discretization error estimates for the discretization of the optimal control problem. We discuss different discretization concepts: full discretization based on piecewise constant functions for the control and full discretization based on piecewise linear and globally continuous functions for the control. Thereby, the state is always approximated by standard linear finite elements. It is well known that the full discretization with piecewise linear functions coincides in the case of no control bounds with the concept of variational discretization, which was introduced in \cite{Hinze:2005} for distributed control problems. Our main convergence results for the discrete controls are as follows: Let $\oq$ be the solution of the optimal control problem \cref{NeumannCon:eq:problem} and $\oq_h$ its discrete counterpart, which either belongs to $Q_h^0$, the space of piecewise constant functions, or to $V_h^\partial$ the space of piecewise linear and globally continuous functions. Then, in \cref{theorem:main1} we show that 
\begin{equation}\label{eq:main1}
\norm{\bar q-\bar q_h}_{L^2(\partial\Omega)}\le
\begin{cases}
	ch\norm{u_d}_{L^2(\Omega)} &\text{if } \bar q_h \in Q_h^0,\\
	ch^{\frac32}\norm{u_d}_{L^2(\Omega)} &\text{if } \bar q_h \in V_h^\partial.
\end{cases}
\end{equation}
Both estimates are optimal if the desired state $u_d$ only belongs to $L^2(\Omega)$. If we assume more regular desired states, such as $u_d\in H^1(\Omega)$, we will show for $\bar q_h\in V_h^\partial$ the improved convergence estimates
\begin{equation}\label{eq:main2}
	\norm{\bar q-\bar q_h}_{L^2(\partial\Omega)}\le
	\begin{cases}
		c h^{\frac32+s}\norm{u_d}_{H^1(\Omega)}&\quad\text{for } 0<s<s_\Omega,\\
		c h^{2}\abs{\ln h}\norm{u_d}_{H^1(\Omega)}&\quad\text{for }\lambda_\Omega >\frac32,
	\end{cases}
\end{equation}
which are again optimal. Moreover, we show optimal convergence estimates for the states: Let $\bar u$ be the optimal state, i.e., the solution to \cref{NeumannCon:eq:state} with right hand side $\bar q$ and let $\bar u_h$ be the optimal discrete state, which solves the corresponding discrete state equation with right hand side $\bar q_h$. Then we show the estimate 
\[
	\norm{\bar u - \bar u_h}_{L^2(\Omega)}\le c h^2 \norm{u_d}_{L^2(\Omega)},
\]
i.e., we obtain a convergence rate of two. Let us emphasize that this estimate holds for any convex and polyhedral domain, for desired states $u_d$, which only belong to~$L^2(\Omega)$, and for both discretization strategies.

Before discussing the existing literature on Neumann boundary control problems, let us briefly comment on possible extensions of the results to problems with pointwise inequality constraints on the control variable. All the convergence results for the discrete controls from above can be extended to problems with pointwise control constraints in a certain way. The convergence rates from \cref{eq:main1} remain valid. In case of $\bar q_h\in V_h^\partial$ only a standard structural assumption on the optimal control $\bar q$ and more regular desired states $u_d\in L^p(\Omega)$ with $p>3$ are required in addition. The estimates of \cref{eq:main2} can be transferred to the concept of variational discretization, which was already mentioned above, and to the postprocessing approach, which was introduced in \cite{MeyerRoesch:2004} for distributed optimal control problems. For the latter approach, one also requires the structural assumption on the optimal control $\bar q$ and more regular desired states $u_d\in L^p(\Omega)$ with $p>3$.
In contrast, the convergence result for the states heavily relies on the fact that no control bounds are present. Its proof is based on a duality argument, where the dual problem is an optimal control problem with a certain desired state in $L^2(\Omega)$, and the estimates of \cref{eq:main1}.

Let us briefly review the existing literature. First results on numerical analysis for Neumann boundary control problems can be found in \cite{Geveci:1979}. A full discretization using piecewise constant functions for the control is considered in \cite{Geveci:1979,CasasMateosTroeltzsch:2005}. A full discretization based on piecewise linear and globally continuous functions for the control is discussed in \cite{CasasMateos:2008}. The concept of variational discretization applied to Neumann boundary control problems can be found in \cite{CasasMateos:2008,HinzeUlrich:2009,Pfefferer:Thesis:2014}. The postprocessing approach applied to Neumann boundary control problems is analyzed in \cite{MateosRoesch:2011,ApelPfeffererRoesch:2015,KrumbiegelPfefferer:2015,Pfefferer:Thesis:2014,ApelWinklerPfefferer:2016,ApelWinklerPfefferer:2018}.
We note that in \cite{CasasMateosTroeltzsch:2005,CasasMateos:2008,KrumbiegelPfefferer:2015,Pfefferer:Thesis:2014} not only linear quadratic Neumann boundary control problems are considered but even problems with semilinear state equation. In \cite{BeuchlerHoferWachsmuthWurst:2015,WachsmuthWurst:2016}, a boundary concentrated finite element method for Neumann boundary control problems is introduced and analyzed. The discussion of Neumann boundary control problems with controls in $H^{-\nhalf}(\partial\Omega)$ and the corresponding numerical analysis can be found in \cite{ApelSteinbachWinkler:2016}.
Most likely, the results from \cite{ApelPfeffererRoesch:2015,Pfefferer:Thesis:2014, KrumbiegelPfefferer:2015,ApelWinklerPfefferer:2018} can be compared best with ours in a certain sense. In those references quasi-optimal discretization error estimates are proven for the post-processing approach employing quasi-optimal finite element error estimates in the $L^2$-norm on the boundary. Moreover, in \cite{Pfefferer:Thesis:2014} quasi-optimal discretization error estimates are established for the concept of variational discretization. Whereas the references \cite{ApelPfeffererRoesch:2015,Pfefferer:Thesis:2014,KrumbiegelPfefferer:2015} are restricted to polygonal domains, in \cite{ApelWinklerPfefferer:2018} polyhedral domains are considered. In those references even gradually refined meshes are used to retain the full order of convergence.
However, in \cite{ApelWinklerPfefferer:2018} it was not explicitly discovered that the convergence rate in convex domains solely depend on the size of the largest interior edge angle and hence corner singularities are not relevant. Moreover, as already mentioned above for the 2D case, the $\log$-factor is slightly worse there compared to our result. We also think that our new approach to prove the finite element error estimates on the boundary might stimulate further research activities in this direction.

The structure of the paper is as follows: In \cref{Sec:2} we introduce the function space setting and establish regularity results for the solutions of the state equations (primal and adjoint), which are needed for the numerical analysis later on. \cref{Sec:3} summarizes various facts for the optimal control problem such as the reduced form, the optimality system and higher regularity of the solution. In \cref{Sec:4} one can find the discretization of the state equations, which is based on linear finite elements, and corresponding discretization error estimates. In particular, optimal discretization error estimates on the boundary are established in various norms for different regularity requirements of the state. Moreover, a best approximation result in weighted norms is shown, which is essential for improved error estimates on the boundary. The discretization of the optimal control problem and corresponding discretization error estimates for the directional derivatives and the optimal solutions are contained in \cref{Sec:5}. Numerical experiments, which underline the theoretical findings of \cref{Sec:5}, are collected in \cref{Sec:6}.

Let us finally note that throughout the whole paper, the constant $c>0$ is a generic constant which is always independent of the discretization parameter $h$.

\section{Regularity results for the state/adjoint equation.}\label{Sec:2}
Throughout the paper the domain $\Omega \subset \R^3$ is assumed to be polyhedral and convex. The set of faces is denoted by ${\cal F}( \Omega)$, the set of all edges by ${\cal E}(\Omega)$, and the set of all vertices by ${\cal V}(\Omega)$. For every edge $e\in {\cal E}(\Omega)$ we denote by $\omega_e$ the interior angle at $e$. We define
\[
\omega_\Omega = \max_{e \in {\cal E}(\Omega)} \omega_e
\]
and have $\omega_\Omega < \pi$ by convexity of $\Omega$. Moreover, we will use the critical exponent $\lambda_\Omega>1$ defined as
\begin{equation}\label{eq:lambda_Omega}
	\lambda_\Omega = \frac{\pi}{\omega_\Omega}.
\end{equation}

We use the standard notation for the Lebesgue $L^2(\Omega)$ and Sobolev spaces $H^{k}(\Omega)$ as well as the corresponding fractional spaces $H^{s}(\Omega)$ equipped with Sobolev-Slobodeckij norms, see, e.g., \cite[Chapter 1.3]{Grisvard:1985} and \cite[Chapter 5]{Adams:2003} for details. We will also use weighted Sobolev spaces with a weight being a power of the distance
to the boundary $\partial \Omega$. The function $x \mapsto \dist(x,\partial \Omega)$ is Lipschitz continuous with the Lipschitz constant $L=1$, but does not in general belong to $W^{2,\infty}(\Omega)$, since the boundary $\partial \Omega$ is polyhedral. For the definition of the weighted spaces this Lipschitz continuity is sufficient, but for weighted superconvergence estimates (see below) we require a smoothed weight. By \cite[Theorem 2, page 171]{Stein:1970} there exists a function $\rho \colon \bar \Omega\to \R_+$ with $\rho \in C^\infty(\Omega)$ fulfilling 
\[
c_1 \rho(x) \le \dist(x,\partial \Omega) \le c_2 \rho(x) \quad \text{for all } x \in \bar \Omega,
\]
with some constants $c_1$ and $c_2$ as well as
\[
\abs{D^\alpha \rho(x)} \le c \rho(x)^{1-\abs{\alpha}} \quad \text{for all } x \in \Omega,
\]
for every multi-index $\alpha \in \N_0^3$. 

For $m \in \N_0$, and a factor $s \in \R$ the space $H^m_s(\Omega)$ is the space of functions with the finite norm
\begin{equation}\label{PDE:eq:norm_weighted_w_space}
	\norm{v}_{H^m_s(\Omega)}^2 = \sum_{\abs{\alpha}\le m} \into \abs{D^\alpha v(x)}^2 \rho(x)^{2s}\, \mathrm{d}x,
\end{equation}
see \cite{MazyaMitreaShaposhnikova:2010} for details. For $s=0$ this space obviously coincides  with $H^m(\Omega)$.

We will use the following Hardy type inequalities.

\begin{proposition}\label{theorem:HardyFractional}
	Let $0<s<\half$. There is a constant $c>0$ such that the following inequality holds
	\[
	\ltwonorm{\rho^{-s} v} \le c \norm{v}_{H^s(\Omega)}
	\]
	for all $v \in H^s(\Omega)$.
\end{proposition}
\begin{proof}
	We refer to \cite[Theorem 1.4.4.3]{Grisvard:1985}, see also \cite[(17)]{Dyda:2004} and \cite[Corollary 2.4]{ChenSong:2003}.
\end{proof}

\begin{proposition}\label{theorem:HardyType_s}
	Let $0<s <\frac{1}{2}$. There is a constant $c>0$ such that the following inequality holds
	\[
	\ltwonorm{\rho^{-s} v} \le \frac{c}{1-2s} \left(\ltwonorm{\rho^{1-s}v}+ \ltwonorm{\rho^{1-s}\nabla v}\right)
	\]
	for all $v \in H^1_{1-s}(\Omega)$.
\end{proposition}
\begin{proof}
	The estimate is given in \cite[Theorem 2.4, page 293]{Necas:2012}. The dependence of the constant on $s$ can be tracked from the proof.
\end{proof}

For a face $F \in \mathcal{F}(\partial \Omega)$ we will also use standard notion of Sobolev spaces $H^{s}(F)$ for $s\in \R$. For $0\le s\le 1$ the definition of these spaces can be directly extended to the whole boundary leading to the spaces  $H^{s}(\partial \Omega)$. The corresponding dual space is denoted by $H^{-s}(\partial \Omega)$.
We will use standard trace theorems, which provide continuity for the trace operator $\tau \colon H^{1}(\Omega) \to H^{\nicefrac{1}{2}}(\partial \Omega)$, see, e.g., \cite[Theorem 1.5.1.3]{Grisvard:1985}. 

In the sequel we require some regularity results to be applied for both state and the adjoint equations. We consider the weak solution $u \in H^1(\Omega)$ of
\begin{equation}\label{eq:z_rhs_f}
	\begin{aligned}
		-\Lap u + u &= f &\quad&\text{in } \Omega,\\
		\partial_n u &=g &\quad&\text{on } \partial \Omega
	\end{aligned}
\end{equation}
for given $f \in L^2(\Omega)$ and $g \in L^2(\partial \Omega)$, i.e.,
\begin{equation}\label{eq:z_rhs_f_weak}
u \in H^1(\Omega) \quad:\quad (\nabla u,\nabla \phi) + (u,\phi) = (f,\phi) + (g,\phi)_{\partial \Omega} \quad \text{for all } \phi \in H^1(\Omega).
\end{equation}
Here, $(\cdot,\cdot)$ denotes the inner product in $L^2(\Omega)$ and $(\cdot,\cdot)_{\partial \Omega}$ denotes the inner product in $L^2(\partial \Omega)$. By standard arguments (see, e.g., the Lax-Milgram Theorem) we obtain the unique solvability of \cref{eq:z_rhs_f_weak} in $H^1(\Omega)$ and the priori estimate
\[
	\norm{u}_{H^1(\Omega)} \le c \left(\ltwonorm{f}+\ltwonormdo{g}\right)
\]
with a constant $c>0$ independent of the data.
\begin{theorem}\label{theorem:regularity_state_adjoint}
	\begin{enumerate}
		\item\label{it:1} Let  $f \in L^2(\Omega)$ and $g \in L^2(\partial \Omega)$. There holds $u\in H^{\nicefrac{3}{2}}(\Omega)$, $\rho^{\nhalf}\nabla^2 u\in L^2(\Omega)$ and
		\[
			\norm{u}_{H^{\nicefrac{3}{2}}(\Omega)}+\ltwonorm{\rho^{\nhalf}\nabla^2 u} \le c \left(\ltwonorm{f}+\ltwonormdo{g}\right)
		\]
		with a constant $c>0$ independent of the data.
		\item\label{it:2} Let  $f \in L^2(\Omega)$ and $g\in H^{\nicefrac{1}{2}}(F)$ for all faces $F\in{\cal F}(\Omega)$. There holds $u \in H^2(\Omega)$ and
		\[
		\htwonorm{u}\le c\left(\ltwonorm{f} + \sum_{F\in{\cal F}(\Omega)}\norm{g}_{H^{\nhalf}(F)}\right)
		\]
		with a constant $c>0$ independent of the data.
	\end{enumerate}
\end{theorem}
\begin{proof}
	\ref{it:1}.
	From \cite[Theorem 2 and Remark~(b) after Theorem 3]{JerisonKenig:1981} we obtain that there is a unique weak solution $u_g$ with $(u_g,1)=0$ of
	\begin{alignat*}{2}
		-\Delta u_g &=0 &&\quad \text{in }\Omega,\\
		\partial_n u_g & = g + \abs{\partial \Omega}^{-1}(f-u,1) && \quad\text{on }\partial\Omega,
	\end{alignat*}
	which belongs to $H^{\nicefrac{3}{2}}(\Omega)$ for $g\in L^2(\partial\Omega)$ and $f\in L^2(\Omega)$ as the compatibility condition
	\begin{equation}\label{eq:solvability}
		(g + \abs{\partial \Omega}^{-1}(f-u,1),1)_{\partial\Omega}=(g,1)_{\partial\Omega}+(f-u,1)=0
	\end{equation}
	is always fulfilled due to \cref{eq:z_rhs_f_weak}. Moreover, according to \cite[Theorem 4.1]{JerisonKenig:1995} and \cite[Theorem 2 and Remark~(b) after Theorem 3]{JerisonKenig:1981} there holds
	\begin{align*}
		\ltwonorm{\rho^{\nhalf}\nabla^2 u_g}&\le c \norm{u_g}_{H^{\nicefrac{3}{2}}(\Omega)}\le c \norm{g+ \abs{\partial \Omega}^{-1}(f-u,1)}_{L^2(\partial\Omega)}\\
		&\le c\left(\norm{f}_{L^2(\Omega)}+\norm{g}_{L^2(\partial\Omega)}\right),
	\end{align*}
	where in the last step we also used the a priori estimate in $H^1(\Omega)$ for $u$.
	Next, let $w\in H^1(\Omega)$ be the unique weak solution of
	\begin{alignat*}{2}
		-\Delta w + w &=f-u_g &&\quad \text{in }\Omega,\\
		\partial_n w & = \abs{\partial \Omega}^{-1}(g,1)_{\partial\Omega} && \quad\text{on }\partial\Omega.
	\end{alignat*}
	By the Sobolev inequality, elliptic $H^2(\Omega)$ regularity from assertion \ref{it:2} of this theorem, and the a priori estimate for $u_g$ we obtain
	\[
		\norm{w}_{H^{\nicefrac{3}{2}}(\Omega)}\le c\htwonorm{w}\le c\left(\ltwonorm{f-u_g}+\abs{(g,1)_{\partial\Omega}}\right)\le c\left(\ltwonorm{f}+ \norm{g}_{L^2(\partial\Omega)}\right)
	\]
	and obviously due to the boundedness of $\rho$ in $\Omega$
	\[
		\ltwonorm{\rho^{\nhalf}\nabla^2 w}\le c \htwonorm{w}\le c \left(\ltwonorm{f}+ \norm{g}_{L^2(\partial\Omega)}\right).
	\]
	By construction we get that $w+u_g$ solves \cref{eq:z_rhs_f_weak}. Consequently, $u=w+u_g$ and the corresponding a priori estimate is fulfilled.
	
	\ref{it:2}. We already know that $u$ belongs to $H^1(\Omega)$. Next, we show that $\nabla u$ belongs to $H^1(\Omega)$ as well. As a consequence, $u$ belongs to $H^2(\Omega)$. In order to show this higher regularity, we apply \cite[Theorem 8.1.10, part 2]{MazyaRossmann:2010}. To this end, we first introduce the function space setting, which is required to make the upcoming steps comprehensible. In contrast to the definition of the weighted Sobolev spaces in \cite[Section 8.1.2]{MazyaRossmann:2010} we do not use different parameters for each corner and each edge but only use one parameter for all corners and one for all edges, which is sufficient for our purposes. For any non-negative integer $l$ and $\beta,\delta\in \R$ we introduce the space $W_{\beta,\delta}^{l,2}(\Omega)$ as the set of functions with finite norm
	\[
	\norm{v}_{W_{\beta,\delta}^{l,2}(\Omega)}=\sum_{\abs{\alpha}\le l}\ltwonorm{r_v^{\beta-\delta-l+\abs{\alpha}}r_e^{\delta}\partial_x^{\alpha}v},
	\]
	where $r_v$ and $r_e$ denote the distance to the set of corners ${\cal V}(\Omega)$ and the set of edges ${\cal E}(\Omega)$, respectively. According to \cite[(3.1.2) on page 90]{MazyaRossmann:2010} this norm is equivalent to the norm of the space $W^{l,2}_{\beta,\delta}(\Omega)$ introduced in \cite[Section 8.1.2]{MazyaRossmann:2010}.
	From \cite[Theorem 8.1.10, part 2]{MazyaRossmann:2010} with $p=2$, $\beta_j=0$, $\delta_k=0$ for all $j,k$, we obtain $\nabla u \in W_{0,0}^{1,2}(\Omega)$ and as a consequence $\nabla u \in H^1(\Omega)$, if the following conditions are satisfied:
	\begin{enumerate}
		\item[(I)] The data $f$ and $g$ fulfill $f\in W_{0,0}^{0,2}(\Omega)$, $g\in W^{\nicefrac{1}{2},2}_{0,0}(F)$ for all faces $F\in{\cal F}(\Omega)$.
		\item[(II)] The closed strip in the complex plane between the lines $\Re\lambda=-1$ and $\Re\lambda = 0$ contains only the eigenvalues $\lambda =-1$ and $\lambda =0$ of certain operator pencil ${\cal U}_v(\lambda)$ associated with the corner $v\in{\cal V}(\Omega)$ (see \cite[Section 8.1.4]{MazyaRossmann:2010} for details) and that the eigenvalue $\lambda=0$ has geometric and algebraic multiplicity $1$.
		\item[(III)] The smallest positive eigenvalues $\Lambda_v$ of the operator pencils ${\cal U}_v(\lambda)$ satisfy $\frac12 < \min(1,\Lambda_v)$.
		\item[(IV)] The inequalities $\max(0,2-\delta_+^{e})<1$ with $\delta_+^{e}=\pi/\omega_e$ hold for all edges $e\in{\cal E}(\Omega)$ (see also \cite[Section 8.3.4]{MazyaRossmann:2010} for the specification of $\delta_+^{e}$).
	\end{enumerate}
	We now check that all of these requirements are fulfilled if $\Omega$ is convex and if $f\in L^2(\Omega)$ and $g\in H^{\nicefrac{1}{2}}(F)$ for all faces $F\in{\cal F}(\Omega)$:
	\begin{enumerate}
		\item[(I)] First, we note that $L^2(\Omega)=W^{0,2}_{0,0}(\Omega)$ (both norms are equal). Moreover, we have $H^1(\Omega)=W^{1,2}_{0,0}(\Omega)$. This holds due to the trivial embedding $W^{1,2}_{0,0}(\Omega)\hookrightarrow H^1(\Omega)$ and the embedding $W^{1,2}_{0,0}(\Omega)\hookleftarrow H^1(\Omega)$ which is a consequence of the Hardy inequality (see e.g. \cite[(2.2.9) on page 35]{MazyaRossmann:2010}). Thus, $H^{\nicefrac{1}{2}}(F)=W^{\nicefrac{1}{2},2}_{0,0}(F)$ for all faces $F\in{\cal F}(\Omega)$ as both spaces are the corresponding natural trace spaces on the faces.
		\item[(II)] This condition is satisfied according to \cite[Section 2.3.1]{KozlovMazyaRossmann:2001}, i.e., we get that the strip $-1<\Re\lambda<0$ is free of eigenvalues of ${\cal U}_v(\lambda)$ and $\lambda=-1$ and $\lambda=0$ are simple eigenvalues with constant eigenfunctions.
		\item[(III)] According to \cite[Section 8.3.5]{MazyaRossmann:2010} the smallest positive eigenvalue $\Lambda_v$ of the operator pencils ${\cal U}_v(\lambda)$ is greater than $1$ if $\Omega$ is convex. Consequently, the inequality $\frac12 < \min(1,\Lambda_v)$ is fulfilled.
		\item[(IV)] As $\Omega$ is convex, we have $\delta_+^{e}=\pi/\omega_e>1$ and therefore $\max(0,2-\delta_+^{e})<1$.
	\end{enumerate}
	The a priori estimate of the second assertion is a consequence of (I) and \cite[Theoren 6.5.4 and Corollary 7.4.2]{MazyaRossmann:2010} together with \cite[Lemma 8.1.2]{MazyaRossmann:2010} and an appropriate partition of unity of the domain.
\end{proof}

\begin{theorem}\label{theorem:weighted_H2reg}
	Let $f \in H^1(\Omega)$ and $g\equiv0$. Moreover, let $s_\Omega$ be defined by
	\begin{equation}\label{eq:s_Omega}
		s_\Omega = \min\left(\lambda_\Omega-1,\half\right)\in\bigg(0,\frac12\bigg],
	\end{equation}
	where $\lambda_\Omega>1$ is the critical exponent \cref{eq:lambda_Omega}.
	Then, for every $0<s<s_\Omega$, the solution $u$ to \cref{eq:z_rhs_f} possesses the regularity $u \in H^2_{-s}(\Omega)$ and the estimate
	\[
	\norm{u}_{H^2_{-s}(\Omega)} \le \frac{c}{1-2s} \norm{f}_{H^1(\Omega)}
	\]
	holds with a constant $c>0$ depending only on $\Omega$.
\end{theorem}
\begin{remark}
	In order to prove this result, we proceed in principal similar to the proof of \cite[Theorem 7]{PfeffererVexler:2024}, which deals with a comparable result, however, for a problem with homogeneous Dirichlet boundary conditions. This, at first glance, minor difference makes the analysis in our case more involved. To be more precise, solutions to problems in polyhedral domains with homogeneous Dirichlet boundary conditions belong to weighted Sobolev spaces with (purely) homogeneous norm with weight being the distance to the singular points on the boundary, i.e., the corners and the edges. In contrast, solutions of problems with homogeneous Neumann boundary conditions require different weighted Sobolev spaces (see also the proof of \cref{theorem:regularity_state_adjoint}).
\end{remark}
\begin{proof}
Let us first elucidate, how the regularity result of the assertion, where we have a weighted regularity with weight being the distance to the boundary, can be traced back to regularity results in weighted Sobolev spaces mentioned above.
Due to the definition of the weighted spaces $H^2_{-s}(\Omega)$, we have
\[
	\norm{u}_{H^2_{-s}(\Omega)} \le \ltwonorm{\rho^{-s} u} + \ltwonorm{\rho^{-s} \nabla u} +\ltwonorm{\rho^{-s} \nabla^2 u}.
\]
We notice that $\rho^{1-s}$ is uniformly bounded in $L^\infty(\Omega)$ (definitely for $s<1/2$). As a consequence, we deduce from \cref{theorem:HardyType_s} (applied to each of the terms on the right hand side) that
\begin{align*}
	\norm{u}_{H^2_{-s}(\Omega)} &\le \frac{c}{1-2s}\left(\htwonorm{u}+\ltwonorm{\rho^{1-s} \nabla^3 u}\right)\\
	&\le \frac{c}{1-2s}\left(\ltwonorm{f}+\ltwonorm{\rho^{1-s} \nabla^3 u}\right),
\end{align*}
where the last step follows from \cref{theorem:regularity_state_adjoint}.
Next, recall that $r_e$ denotes the distance function to the set of edges ${\cal E}(\Omega)$. As $\rho(x)\le r_e(x)$ for all $x\in\Omega$ and $1-s>0$, we get for the term involving third derivatives that
\[
	\ltwonorm{\rho^{1-s} \nabla^3 u}\le \ltwonorm{r_e^{1-s} \nabla^3 u}.
\]
To bound this term, we will apply \cite[Theorem 8.1.7]{MazyaRossmann:2010}.
However, this theorem is not directly applicable, as it requires, in our case, the solution to be zero at corner points, which does not hold in general for our problem. Instead, we will apply the results to a modified solution where we can guarantee that this modified solution is zero at corner points. 
Let $\eta\colon\R_+ \to[0,1]$ be a smooth cut off function which is equal to one in a neighborhood of zero and which has a support in $[0,1)$. Further, we define smooth cut off functions $\chi_v$ for each corner $v\in{\cal V}(\Omega)$ by
\[
	\chi_v(x)=\eta(\abs{x-v}).
\]
Without loss of generality we assume that the support of each function $\chi_v$ is contained in $\Omega$ and $\chi_v$ vanishes in the neighborhood of any other corner. There holds
\[
	\nabla \chi_v(x)=\frac{\eta'(\abs{x-v})}{\abs{x-v}}(x-v)\quad \text{and} \quad \partial_n\chi_v \equiv 0 \text{ on }\partial\Omega.
\]
Next, we introduce the modified solution. We define the function $\tilde u$ by
\[
	\tilde u = u + u_\text{v}\quad\text{with}\quad u_\text{v}=- \sum_{v\in{\cal V}(\Omega)} u(v)\chi_v,
\]
which solves
\begin{alignat*}{2}
	-\Delta \tilde u+\tilde u  &=\tilde f&&\quad\text{in }\Omega, \\
	\partial_n\tilde u &=\tilde g&&\quad\text{on }\partial\Omega.
\end{alignat*}
with
\[
	\tilde f = f -\Delta u_\text{v}+u_\text{v}\quad \text{and}\quad \tilde g=\partial_n u+\partial_n u_\text{v}=\partial_n u- \sum_{v\in{\cal V}(\Omega)} u(v)\partial_n\chi_v=0.
\]
By the Sobolev inequality and \cref{theorem:regularity_state_adjoint}, we get
\begin{align*}
	\ltwonorm{r_e^{1-s} \nabla^3 u_\text{v}}&\le \sum_{v\in{\cal V}(\Omega)} \abs{u(v)}\, \ltwonorm{r_e^{1-s} \nabla^3 \chi_v}\\
	&\le c\norm{u}_{L^\infty(\Omega)}\le c\norm{u}_{H^2(\Omega)}\le c\ltwonorm{f}.
\end{align*}
As a consequence, there holds
\[
	\ltwonorm{r_e^{1-s} \nabla^3 u}\le \ltwonorm{r_e^{1-s} \nabla^3 \tilde u}+ \ltwonorm{r_e^{1-s} \nabla^3 u_\text{v}} \le \ltwonorm{r_e^{1-s} \nabla^3 \tilde u}+c\ltwonorm{f}.
\]
In summary, we still have to bound third derivatives in a weighted $L^2(\Omega)$-norm. However, in contrast to before, the function $\tilde u$ is equal to zero at each corner point and we can apply \cite[Theorem 8.1.7]{MazyaRossmann:2010} to get the desired term bounded. This is done next.
Let us note that the function space setting, which we need for the subsequent steps, has already been introduced in the proof of \cref{theorem:regularity_state_adjoint}.
From \cite[Theorem 8.1.7]{MazyaRossmann:2010} with $I_0=\emptyset$, $p,q=2$, $l=3$, $\beta_j'=1-s$, $\delta_k'=1-s$, $\beta_j=\beta\in(-1,-1/2)$, $\delta_k=\delta\in(-1,0)$ for all $j,k$,
we now obtain $\tilde u \in W_{1-s,1-s}^{3,2}(\Omega)$ and especially the boundedness of $r_e^{1-s} \nabla^3 \tilde u$ in $L^2(\Omega)$ if the following conditions are satisfied:
\begin{enumerate}
	\item[(I)] The solution $\tilde u$ belongs to $W^{1,2}_{\beta,\delta}(\Omega)$.
	\item[(II)] The right hand side $\tilde f$ fulfills $\tilde f\in W_{1-s,1-s}^{1,2}(\Omega)$.
	\item[(III)] The closed strip between the lines $\Re\lambda=-\beta-\frac12$ and $\Re\lambda =\frac12+s$ is free of eigenvalues of the operator pencil ${\cal U}_v (\lambda)$ (see the proof of \cref{theorem:regularity_state_adjoint} for details).
	\item[(IV)] The inequalities 
		\begin{equation}\label{eq:max_cond}
			\max(0,1-\delta_+^{e})<\delta+1<1\quad\text{and}\quad\max(0,3-\delta_+^{e})<2-s<3
		\end{equation}
	 with $\delta_+^{e}=\pi/\omega_e$ hold for all edges $e\in{\cal E}(\Omega)$.
\end{enumerate}
We now check that all of these requirements are fulfilled if $f\in H^1(\Omega)$ and if $\Omega$ is convex.
\begin{enumerate}
	\item[(I)] There holds
	\[
	 \ltwonorm{-\Delta u_\text{v} + u_\text{v}}\le \sum_{v\in{\cal V}(\Omega)} \abs{u(v)} \ltwonorm{-\Delta \chi_v + \chi_v} \le  c\sum_{v\in{\cal V}(\Omega)} \abs{u(v)}\le c\norm{u}_{L^\infty(\Omega)}.
	\]
	Consequently, the function $\tilde f=f -\Delta u_\text{v}+u_\text{v}$ fulfills due to the Sobolev inequality and elliptic regularity of $u$ (see \cref{theorem:regularity_state_adjoint})
	\[
		\norm{\tilde f}_{L^2(\Omega)}\le \norm{f}_{L^2(\Omega)}+c\norm{u}_{L^\infty(\Omega)}\le \norm{f}_{L^2(\Omega)}+c\norm{u}_{H^2(\Omega)}\le c \norm{f}_{L^2(\Omega)}.
	\]
	Hence, by elliptic regularity of $\tilde u$ (see \cref{theorem:regularity_state_adjoint}), we obtain
	\[
		\norm{\tilde u}_{H^2(\Omega)}\le c\|f\|_{L^2(\Omega)}.
	\]
	By \cite[Lemma 8.1.2]{MazyaRossmann:2010} with arbitrary $\beta,\delta >-1$ and by Hardy's inequality (see e.g. \cite[(2.2.9) on page 35]{MazyaRossmann:2010}) together with the fact that $\tilde u$ vanishes at corner points, we deduce
	\[
		\norm{\tilde u}_{W^{1,2}_{\beta,\delta}(\Omega)}\le c\norm{\tilde u}_{W^{2,2}_{0,0}(\Omega)}\le c\norm{\tilde u}_{H^2(\Omega)}\le c\|f\|_{L^2(\Omega)}
	\]
	such that $\tilde u$ belongs to $W^{1,2}_{\beta,\delta}(\Omega)$.
	\item[(II)] Due to Hardy's inequality (see e.g. \cite[(2.2.9) on page 35]{MazyaRossmann:2010}) we have
	\begin{align*}
		\norm{\tilde f}_{W_{1-s,1-s}^{1,2}(\Omega)}
		&\le c \left(\norm{\tilde f}_{H^1(\Omega)}+\ltwonorm{r_e^{1-s}\nabla \tilde f}\right)\le c \norm{\tilde f}_{H^1(\Omega)},
	\end{align*}
	where we also used the boundedness of $r_{e}^{1-s}$.
	Similar to before, the Sobolev inequality and \cref{theorem:regularity_state_adjoint} then imply
	\[
		\norm{\tilde f}_{W_{1-s,1-s}^{1,2}(\Omega)} \le c\left(\norm{f}_{H^1(\Omega)}+\norm{u}_{L^\infty(\Omega)}\right)\le c\left(\norm{f}_{H^1(\Omega)}+\norm{u}_{H^2(\Omega)}\right)
		\le c\norm{f}_{H^1(\Omega)}.
	\]
	\item[(III)] As we have seen in the proof of \cref{theorem:regularity_state_adjoint} the smallest positive eigenvalue $\Lambda_v$ of the operator pencil ${\cal U}_v(\lambda)$ is greater than $1$ if $\Omega$ is convex. Thus, the closed strip between the lines $\Re\lambda=-\beta-\frac12$ and $\Re\lambda =\frac12+s$ is free of eigenvalues as $\beta\in (-1,-1/2)$ and $s\in (0,1/2)$.
	\item[(IV)] Due to the convexity of the domain ($\pi/\omega_{e}>1$ for all $e\in\mathcal{E}(\Omega)$) the first condition in \cref{eq:max_cond} trivially holds for $\delta\in(-1,0)$.
	For $0<s<s_\Omega$ we get
	\[
		\max\left(0,3-\frac\pi\omega_e\right)< \max\left(\frac32,3-\frac\pi\omega_e\right)=2-\min\left(\frac12,\frac\pi\omega_e-1\right)\le2-s_\Omega<2-s<3,
	\]
	which is the second condition in \cref{eq:max_cond}.
\end{enumerate}
Finally, the a priori estimate of the assertion is a consequence of the previous inequalities and \cite[Theorem 6.5.4 and Corollary 7.4.2]{MazyaRossmann:2010} together with an appropriate partition of unity of the domain.
\end{proof}

\section{Optimal control problem.}\label{Sec:3}
In this section we discuss the optimal control problem \eqref{NeumannCon:eq:problem} and provide optimality conditions. If not explicitly mentioned we make the minimal assumption $u_d \in L^2(\Omega)$. 

We call $Q = L^2(\partial \Omega)$ the control space and introduce the solution operator $S \colon Q \to H^1(\Omega)$ mapping a control $q \in Q$ to the weak solution $u = u(q)$ of \cref{eq:z_rhs_f} with $g = q$ and $f=0$. Using this solution operator we define the reduced cost functional
\[
j \colon Q \to \R, \quad j(q) = J(q,Sq).
\]
Thus, we formulate the optimal control problem as:
\begin{equation}\label{NeumannCon:eq:reduced_problem}
	\text{Minimize } j(q), \quad q \in Q.
\end{equation}
It is straightforward to check, that the reduced cost functional $j$ is continuous and strictly convex. Thus, the existence and uniqueness of an optimal solution $\oq \in Q$ follows by standard arguments. We refer to $\ou = S \oq$ as the optimal state.

The functional $j$ is two times Fr\'echet differentiable. For a control $q \in Q$ and a direction $\dq \in Q$ the directional derivatives are
given by
\[
j'(q)(\dq) = (u-u_d,\du) + \alpha(q,\dq)_{\partial\Om}
\]
and
\[
j''(q)(\dq,\dq) = \ltwonorm{\du}^2 + \alpha \ltwonormdo{\dq}^2,
\]
where $u = Sq$ and $\du = S\dq$. 

The necessary optimality condition for \eqref{NeumannCon:eq:reduced_problem} is given as 
\begin{equation}\label{eq:opt_cond}
	\oq \in Q \quad:\quad	j'(\oq)(\dq) = 0 \quad \text{for all }\dq \in Q.
\end{equation}
By convexity of $j$ this condition is also sufficient for the optimality.

To derive the optimality system we require an adjoint based representation of $j'(q)(\dq)$. To this end we introduce the corresponding adjoint equation.  For a given control $q\in Q$ and corresponding state $u = u(q) \in H^1(\Omega)$
we introduce the adjoint state $z = z(q) \in H^1(\Omega)$ as the weak solution of the adjoint equation
\[
\begin{aligned}
	-\Lap z + z &= u-u_d &&\quad\text{in } \Omega,\\
	\partial_n z &=0 &&\quad\text{on } \partial \Omega,
\end{aligned}
\]
with the corresponding weak formulation
\begin{equation}\label{NeumannCon:eq:adjoint_weak}
	z \in H^1(\Omega) \quad:\quad (\nabla \phi,\nabla z) + (\phi,z) = (u-u_d,\phi) \quad \text{for all } \phi \in H^1(\Omega).
\end{equation}
By $H^2$ regularity from \cref{theorem:regularity_state_adjoint} we get $z \in H^2(\Omega)$. A straightforward calculation provides an expressions for the directional derivative of $j$ based on the solution of this adjoint equation:
\[
j'(q)(\dq) = (\alpha q + z|_{\partial \Omega},\dq)_{\partial \Omega},
\]
where  $q,\dq \in Q$, $u=u(q)\in H^1(\Omega)$ and $z=z(q)\in H^2(\Omega)$ solves the adjoint equation \eqref{NeumannCon:eq:adjoint_weak}. This consideration results in the following optimality system.

\begin{theorem}\label{NeumannCon:theorem:opt_sys}
	A control $\oq\in Q$ is the solution of the optimal control problem
	\eqref{NeumannCon:eq:reduced_problem} if and only if the triple $(\oq,\ou,\oz)$
	with $\oq \in Q$, $\ou \in H^1(\Omega)$  and $\oz \in H^2(\Omega)$ fulfills the following system:
	\begin{subequations}\label{NeumannCon:eq:opt_sys}
		\begin{itemize}
			\item State equation:
			\begin{equation}\label{NeumannCon:eq:opt_sys_state}
				(\nabla \ou,\nabla\phi) + (\ou,\phi)=(\oq,\phi)\quad\text{for all }\phi\in H^1(\Omega)
			\end{equation}
			\item Adjoint equation:
			\begin{equation}\label{NeumannCon:eq:opt_sys_adjoint}
				(\nabla \phi,\nabla \oz) + (\phi,\oz) =(\ou-u_d,\phi) \quad\text{for all }\phi\in H^1(\Omega)\\
			\end{equation}
			\item Optimality condition:
			\begin{equation}\label{NeumannCon:eq:opt_sys_gradient}
				\alpha q + z|_{\partial \Omega} = 0.
			\end{equation}
		\end{itemize}
	\end{subequations}
	Moreover, there holds $\oq \in H^{1}(\partial\Omega)$ and there exists a constant $c>0$ depending only on $\alpha$ and $\Omega$ such that
	\[
	\norm{\oq}_{H^{1}(\partial \Omega)} + \norm{\ou}_{H^2(\Omega)} +  \norm{\oz}_{H^2(\Omega)} \le c \ltwonorm{u_d}.
	\]
\end{theorem}
\begin{proof}
	It only remains to prove the regularity result. As already discussed before the optimal adjoint state belongs to $H^2(\Omega)$. Due to the trace estimate from \cite[Theorem 4.11]{Necas:2012} and \cref{NeumannCon:eq:opt_sys_gradient} this implies $\bar q\in H^1(\partial\Omega)$ and consequently $\bar u\in H^2(\Omega)$ according to the elliptic regularity from \cref{theorem:regularity_state_adjoint}.
\end{proof}

\section{Discretization of the state/adjoint equation.}\label{Sec:4}
For discretization we consider the space of linear finite elements $\widehat V_h \subset H^1(\Omega)$ defined on a mesh $\T_h$ from a family of shape regular quasi-uniform meshes, see, e.g., \cite{BrennerScott:2008}. The mesh $\T_h = \{K\}$ consists of cells $K$, which are open tetrahedrons. Moreover, we let $\T_h^\partial=\{F\}$ be the mesh on the boundary of $\Omega$ consisting of open triangles $F$ (the faces of the corresponding tetrahedrons), which is naturally induced by $\T_h$.
We also use the space $V_h^\partial$ of traces of functions from $\widehat V_h$, i.e.
\begin{equation}\label{FEM:eq:V_h^partial}
V_h^\partial = \Set{\tau v_h | v_h \in \widehat V_h},
\end{equation}
where $\tau \colon H^1(\Omega) \to H^{\nicefrac{1}{2}}(\partial \Omega)$ is the trace operator. 

The equation \eqref{eq:z_rhs_f} is discretized as follows. For given $g \in L^2(\partial \Omega)$ and $f \in L^2(\Omega)$ the discrete solution $u_h \in \widehat V_h$ fulfills
\begin{equation}\label{FEM:eq:poisson_inhomogenNeumann}
	u_h \in  \widehat V_h \quad:\quad (\nabla u_h,\nabla \phi_h) + (u_h,\phi_h) = (f,\phi_h) + (g,\phi_h)_{\partial \Omega} \quad \text{for all } \phi_h \in \widehat V_h.
\end{equation}

We define the Neumann Ritz projection $R_h \colon H^1(\Omega) \to \widehat V_h$ by the following orthogonality relation:
\[
(\nabla (u-R_h u),\nabla \phi_h) + (u-R_h u,\phi_h) = 0 \quad \text{for all } \phi_h \in \widehat V_h.
\]
Thus, for $u$ solving \eqref{eq:z_rhs_f} we have that $u_h = R_h u$ is the solution of \eqref{FEM:eq:poisson_inhomogenNeumann}.
 
The following standard estimate requires only $H^2(\Omega)$ regularity of the solution $u$ of~\eqref{eq:z_rhs_f}.
\begin{theorem}\label{theorem:standard_error_est_H2}
Let $f\in L^2(\Omega)$, $g\in H^{\nicefrac{1}{2}}(F)$ for all faces $F\in{\cal F}(\Omega)$ and $u$ be the solution of \eqref{eq:z_rhs_f}. Let $u_h \in \widehat V_h$ be the solution of \eqref{FEM:eq:poisson_inhomogenNeumann}. Then the following estimates hold
\begin{align*}
\ltwonorm{u-u_h}&\le ch \honenorm{u-u_h}\le ch \inf_{v_h\in\widehat V_h}\honenorm{u-v_h}\\
&\le ch^2 \left(\ltwonorm{f} + \sum_{F\in{\cal F}(\Omega)}\norm{g}_{H^{\nhalf}(F)}\right)
\end{align*}
with a constant $c>0$ only depending on $\Omega$.
\end{theorem}
\begin{proof}
	This is a standard result in the theory of finite element methods using the regularity result from \cref{theorem:regularity_state_adjoint}, see, e.g.,~\cite[Theorem (5.7.6) and eq. (5.7.5)]{BrennerScott:2008}.
\end{proof} 

Based on the finite element error estimate, the following stability estimate can be shown.

\begin{lemma}\label{theorem:L2stability}
	Let $u_h$ solve \cref{FEM:eq:poisson_inhomogenNeumann} with $f\equiv0$ and $g\in L^2(\partial\Omega)$. There is the estimate
	\[
		\norm{u_h}_{L^2(\Omega)}\le c \left(h\norm{g}_{H^{-\nicefrac12}(\partial\Omega)}+\norm{g}_{H^{-1}(\partial\Omega)}\right)
	\]
	with a constant $c>0$ only depending on $\Omega$.
\end{lemma}
\begin{proof}
	Let $w\in H^1(\Omega)$ be the weak solution of
	\[
	\begin{aligned}
		-\Lap w + w &= u_h &&\quad\text{in } \Omega,\\
		\partial_n w &=0 &&\quad\text{on } \partial \Omega,
	\end{aligned}
	\]
	and let $w_h=R_hw\in \widehat V_h$ its Neumann Ritz projection.
	There holds
	\begin{align*}
			\norm{u_h}^2_{L^2(\Omega)}&=(\nabla w_h,\nabla u_h) + (w_h,u_h)=(g,w_h)_{\partial \Omega}=(g,w_h-w)_{\partial \Omega}+(g,w)_{\partial \Omega}\\
			&\le \norm{g}_{H^{-\nicefrac12}(\partial\Omega)}\norm{w-w_h}_{H^{\nicefrac12}(\partial\Omega)}+\norm{g}_{H^{-1}(\partial\Omega)}\norm{w}_{H^{1}(\partial\Omega)}.
	\end{align*}
	The standard trace estimate and \cref{theorem:standard_error_est_H2}
	yield
	\[
		\norm{w-w_h}_{H^{\nicefrac12}(\partial\Omega)}\le c \norm{w-w_h}_{H^{1}(\Omega)}\le ch \norm{u_h}_{L^2(\Omega)}.
	\]
	Using the trace estimate from \cite[Theorem 4.11]{Necas:2012} and part 2 of \cref{theorem:regularity_state_adjoint}, we get
	\[
		\norm{w}_{H^{1}(\partial\Omega)} \le c \norm{w}_{H^{2}(\Omega)} \le c\norm{u_h}_{L^2(\Omega)}.
	\]
\end{proof}

\begin{lemma}\label{FEM:lemma:u-uhboundary_negative}
	Let $f \in L^2(\Omega)$, $g=0$ and $u$ be the solution of \eqref{eq:z_rhs_f}. Moreover, let $u_h = R_h u$ be its Neumann Ritz projection. Then there holds
	\[
	\norm{u-u_h}_{H^{-\nicefrac{1}{2}}(\partial\Omega)}+h^{\frac12}\norm{u-u_h}_{L^2(\partial\Omega)} \le c h^{2} \norm{f}_{L^2(\Omega)}
	\]
	with a positive constant $c$ only depending on $\Omega$.
\end{lemma}
\begin{proof}
	Let us start with the $L^2(\partial\Omega)$-estimate. This is a direct consequence of the trace theorem from \cite[Theorem 1.6.6]{BrennerScott:2008} and the finite element error estimate from \cref{theorem:standard_error_est_H2}, i.e.,
	\[
	\norm{u-u_h}_{L^2(\partial\Omega)}\le \norm{u-u_h}_{L^2(\Omega)}^{\frac12} \norm{u-u_h}_{H^1(\Omega)}^{\frac12} \le c h^{\frac32} \norm{f}_{L^2(\Omega)}.
	\]
	For the second estimate, let $w\in H^1(\Omega)$ be the weak solution of
	\[
	\begin{aligned}
		-\Lap w + w &= 0 &&\quad\text{in } \Omega,\\
		\partial_n w &=\delta &&\quad\text{on } \partial \Omega,
	\end{aligned}
	\]
	with some $\delta\in H^{\nicefrac{1}{2}}(\partial\Omega)$ and let $w_h=R_hw\in \widehat V_h$ its Neumann Ritz projection. There holds due the Galerkin orthogonality and \cref{theorem:standard_error_est_H2}
	\begin{align*}
		(\delta,u-u_h)_{\partial\Omega}&=(\nabla (w-w_h),\nabla(u-u_h))+(w-w_h,u-u_h)\\
		&\le \honenorm{w-w_h}\honenorm{u-u_h}\\
		&\le c h^2 \norm{\delta}_{H^{\nicefrac{1}{2}}(\partial\Omega)}\ltwonorm{f}.
	\end{align*}
	As a consequence we obtain
	\[
	\norm{u-u_h}_{H^{-\nicefrac{1}{2}}(\partial\Omega)}=\sup_{\delta\in H^{\nicefrac{1}{2}}(\partial\Omega),\, \delta \not\equiv 0}\frac{(\delta,u-u_h)_{\partial\Omega}}{\norm{\delta}_{H^{\nicefrac{1}{2}}(\partial\Omega)}} \le c h^2 \norm{f}_{L^2(\Omega)}.
	\]
\end{proof}
The latter result is optimal in case that $H^2(\Omega)$ regularity of the solution $u$ is available only. If the solution is more regular, a higher convergence rate is possible in $L^2(\partial\Omega)$. In order to show the improved error estimates we require a best approximation result with respect to a weighted norm. To this end we define a regularized weight $\tilde \rho \colon \bar \Omega \to \R_+$ by
\begin{equation}\label{eq:tilde_rho}
	\quad \tilde \rho(x) = \sqrt{\rho(x)^2+h^2}.
\end{equation} 
By a direct calculation we obtain 
\begin{equation}\label{eq:prop_tilde_rho}
	\abs{\nabla \tilde \rho}  \le c \quad \text{and} \quad \abs{\nabla^2 \tilde \rho} \le c \tilde \rho^{-1}.
\end{equation}
A direct consequence of the boundedness of the gradient of $\tilde \rho$ is the existence of a constant $c$ independent of $h$ such that the following estimate hold
\begin{equation}\label{eq:tilde_rho_local}
	\max_{x \in \bar K} \tilde \rho(x) \le c \min_{x \in \bar K} \tilde \rho (x) \quad \text{for all } K \in \T_h.
\end{equation}
This leads to weighted interpolation estimates for the Lagrange interpolation $i_h \colon C(\bar \Omega) \to \widehat V_h$.
\begin{lemma}\label{FEM:lemma:est_interpolation_general_weight}
	Let $\alpha \in \R$ be arbitrary. Then there is a constant $c>0$ independent of $h$ such that
	\[
	\ltwonorm{\tilde \rho^\alpha(v-i_h v)} + h \ltwonorm{\tilde \rho^\alpha\nabla (v-i_h v)} \le c h^2 \ltwonorm{\tilde \rho^\alpha \nabla^2 v}
	\]
	holds for all $v \in H^2(\Omega)$.
\end{lemma}

The next result is a superconvergence type estimate, see \cite[Lemma 4]{PfeffererVexler:2024}. A similar result for another weight  (regularized distance to a point) can be found, e.g., in \cite[Lemma 3]{LeykekhmanD_VexlerB_2016c}. 
\begin{lemma}\label{FEM:lemma:superconvergence_ih}
	Let $\alpha,\beta \in \R$ be arbitrary. There is a constant $c>0$ independent of $h$ such that	
	\[
	\ltwonorm{\tilde \rho^\alpha(\operatorname{Id}-i_h)(\tilde \rho^\beta v_h)} + h\ltwonorm{\tilde \rho^\alpha\nabla  (\operatorname{Id}-i_h)(\tilde \rho^\beta v_h)} \le c h \ltwonorm{\tilde \rho^{\alpha+\beta-1}v_h}
	\]
	holds for all $v_h \in \widehat V_h$.
\end{lemma}

The next lemma provides a weighted stability result for the Neumann Ritz projection.

\begin{lemma}\label{FEM:lemma:Neumann_weighted_stabilty}
	Let $u \in H^1(\Omega)$ and $u_h = R_h u \in \widehat V_h$ be its Neumann Ritz projection. There holds
	\[
	\ltwonorm{\tilde \rho^{\half}u_h}+\ltwonorm{\tilde \rho^{\half}\nabla u_h} \le c \left(\ltwonorm{\tilde \rho^{\half}\nabla u} + \ltwonorm{\tilde \rho^{\half}u}+ \ltwonorm{\tilde \rho^{-\half}u_h}\right)
	\]
	with a constant $c>0$ independent of $h$.
\end{lemma}
\begin{proof}
	We obtain by introducing intermediate terms
	\[
	\begin{aligned}
		&\ltwonorm{\tilde \rho^{\half}u_h}^2+\ltwonorm{\tilde \rho^{\half}\nabla u_h}^2=(\tilde \rho u_h,u_h)+(\nabla(\tilde \rho u_h),\nabla u_h) - (u_h\nabla \tilde \rho,\nabla u_h)\\
		&\quad=((\operatorname{Id}-i_h)(\tilde \rho u_h),u_h)+(\nabla(\operatorname{Id}-i_h)(\tilde \rho u_h),\nabla u_h)+(i_h(\tilde \rho u_h),u_h)+(\nabla i_h(\tilde \rho u_h),\nabla u_h)- (u_h\nabla \tilde \rho,\nabla u_h)\\
		&\quad = I_1+I_2+I_3+I_4+I_5.
	\end{aligned}
	\]
	For $I_1+I_2$ we get with \cref{FEM:lemma:superconvergence_ih}
	\begin{align*}
		I_1+I_2
		&\le\ltwonorm{\tilde \rho^{-\half}(\operatorname{Id}-i_h)(\tilde \rho u_h)}\ltwonorm{\tilde \rho^{\half}u_h}+\ltwonorm{\tilde \rho^{-\half}\nabla (\operatorname{Id}-i_h)(\tilde \rho u_h)}\ltwonorm{\tilde \rho^{\half}\nabla u_h}\\
		&\le c\left(h\ltwonorm{\tilde \rho^{-\half}u_h}\ltwonorm{\tilde \rho^{\half}u_h}+\ltwonorm{\tilde \rho^{-\half}u_h}\ltwonorm{\tilde \rho^{\half}\nabla u_h}\right)\\
		&\le c\ltwonorm{\tilde \rho^{-\half}u_h}^2 + \frac14 \left(\ltwonorm{\tilde \rho^{\half}u_h}^2+\ltwonorm{\tilde \rho^{\half}\nabla u_h}^2\right).
	\end{align*}
	By definition of the Neumann Ritz projection and \cref{FEM:lemma:superconvergence_ih} we have for $I_3+I_4$
	\begin{align*}
		I_3+I_4&=(i_h(\tilde \rho u_h),u)+(\nabla i_h(\tilde \rho u_h),\nabla u)\\
		&=((i_h-\operatorname{Id})(\tilde \rho u_h),u)+(\nabla (i_h-\operatorname{Id})(\tilde \rho u_h),\nabla u)+(\tilde \rho u_h,u)+(\nabla (\tilde \rho u_h),\nabla u)\\
		&\le \left(\ltwonorm{\tilde \rho^{-\half}(\operatorname{Id}-i_h)(\tilde \rho u_h)}+\ltwonorm{\tilde \rho^{\half}u_h}\right)\ltwonorm{\tilde \rho^{\half}u}\\
		&\quad+\left(\ltwonorm{\tilde \rho^{-\half}\nabla (\operatorname{Id}-i_h)(\tilde \rho u_h)}+\ltwonorm{\tilde \rho^{-\half}\nabla (\tilde\rho u_h)}\right)\ltwonorm{\tilde \rho^{\half}\nabla u}\\
		&\le \left(ch\ltwonorm{\tilde \rho^{-\half}u_h}+\ltwonorm{\tilde \rho^{\half}u_h}\right)\ltwonorm{\tilde \rho^{\half}u}\\
		&\quad + \left(c\ltwonorm{\tilde \rho^{-\half}u_h}+\ltwonorm{\tilde\rho^{-\half}u_h\nabla \tilde\rho }+\ltwonorm{\tilde \rho^{\half}\nabla u_h}\right)\ltwonorm{\tilde \rho^{\half}\nabla u}.
	\end{align*}
	Since $\tilde \rho \ge h$ and $\abs{\nabla \tilde \rho} \le c$ we further deduce
	\begin{align*}
		I_3+I_4&\le c\ltwonorm{\tilde \rho^{\half}u_h}\ltwonorm{\tilde \rho^{\half}u} + c\left(\ltwonorm{\tilde \rho^{-\half}u_h}+\ltwonorm{\tilde \rho^{\half}\nabla u_h}\right)\ltwonorm{\tilde \rho^{\half}\nabla u}\\
		&\le c\left( \ltwonorm{\tilde \rho^{\half}\nabla u}^2+\ltwonorm{\tilde \rho^{\half}u}^2+ \ltwonorm{\tilde \rho^{-\half}u_h}^2\right)+\frac14 \left(\ltwonorm{\tilde \rho^{\half}u_h}^2+\ltwonorm{\tilde \rho^{\half}\nabla u_h}^2\right).
	\end{align*}
	For $I_5$ we obtain by $\abs{\nabla \tilde \rho} \le c$
	\[
		I_5\le \ltwonorm{\tilde\rho^{-\half}u_h\nabla \tilde\rho }\ltwonorm{\tilde \rho^{\half}\nabla u_h}\le c\ltwonorm{\tilde\rho^{-\half}u_h}^2 + \frac14\ltwonorm{\tilde \rho^{\half}\nabla u_h}^2
	\]
	Putting terms together and absorbing the terms involving $\tilde \rho^{\half}u_h$ and $\tilde \rho^{\half}\nabla u_h$ into the left-hand side we complete the proof.
\end{proof}

The next theorem provides a best approximation result in weighted norms, which will be used in the sequel, but seems also to be of independent  interest.

\begin{theorem}\label{FEM:Neumann:weighted_best_approximation}
	Let $u \in H^1(\Omega)$ and $u_h = R_h u \in \widehat V_h$ be its Neumann Ritz projection. There is a constant $c>0$ independent of $h$ such that
	\begin{multline}
		\ltwonorm{\tilde \rho^{-\half}(u-u_h)} + \ltwonorm{\tilde \rho^{\half}\nabla (u-u_h)}\\ \le c\left(\ltwonorm{\tilde \rho^{\half}\nabla (u-v_h)} + \ltwonorm{\tilde \rho^{-\half}(u-v_h)}\right)
	\end{multline}
	for every $v_h \in \widehat V_h$.
\end{theorem}
\begin{proof}
	We first observe that for every $v_h\in\widehat{V}_h$ there holds
	\begin{align*}
		\ltwonorm{\tilde \rho^{\half}\nabla (u-u_h)} &\le \ltwonorm{\tilde \rho^{\half}\nabla (u-v_h)} + \ltwonorm{\tilde \rho^{\half}\nabla (v_h-u_h)}.
	\end{align*}
	Application of \cref{FEM:lemma:Neumann_weighted_stabilty} with $v_h-u_h = R_h(v_h-u)$ results in the following estimate for the last term:
	\begin{align*}
		\ltwonorm{\tilde \rho^{\half}\nabla (v_h-u_h)}&\le c\left(\ltwonorm{\tilde \rho^{\half}\nabla (u-v_h)} + \ltwonorm{\tilde \rho^{\half}(u-v_h)}+\ltwonorm{\tilde \rho^{-\half}(v_h-u_h)}\right)\\
		&\le c\left(\ltwonorm{\tilde \rho^{\half}\nabla (u-v_h)} + \ltwonorm{\tilde \rho^{-\half}(u-v_h)}+\ltwonorm{\tilde \rho^{-\half}(u-u_h)}\right),
	\end{align*}
	where we inserted $u$ as an intermediate function and employed $\max_{x\in\Omega}\abs{\tilde \rho}\le c$.
	It remains to appropriately estimate $\ltwonorm{\tilde \rho^{-\half}(u-u_h)}$. To this end, we simply use the fact that $(\tilde \rho(x))^{-1}\le h^{-1}$ for all $x\in\Omega$ and \cref{theorem:standard_error_est_H2}. This yields
	\begin{align*}
		\ltwonorm{\tilde \rho^{-\half}(u-u_h)}&\le ch^{-\half}\ltwonorm{u-u_h}\le ch^{\half}\honenorm{u-v_h}\\
		&\le c\left(\ltwonorm{\tilde \rho^{\half}\nabla (u-v_h)} + \ltwonorm{\tilde \rho^{-\half}(u-v_h)}\right),
	\end{align*}
	where we also used $\max_{x\in\Omega}\abs{\tilde \rho}\le c$ in the last step. Collecting the different estimates ends the proof.
\end{proof}

The next theorem provides an improved error estimate for $u-u_h$ on the boundary~$\partial \Omega$ under the regularity assumption $u \in H^2_{-s}(\Omega)$.
\begin{theorem}\label{FEM:Theorem:u-uhboundary}
Let $u \in H^1(\Omega)$ possess additional weighted regularity $u \in H^2_{-s}(\Omega)$ for some $0<s<\half$. Let $u_h = R_h u \in \widehat V_h$ be its Neumann Ritz projection. Then there holds the following estimate
\[
\ltwonormdo{u-u_h} \le c h^{\frac{3}{2}+s} \norm{u}_{H^2_{-s}(\Omega)}
\]
with a constant $c>0$ only depending on $\Omega$.
\end{theorem}
\begin{proof}
	 We apply a duality argument. Let $z\in H^1(\Omega)$ be the solution of
	\[
	(\nabla \phi,\nabla z) + (\phi,z) = (u-u_h,\phi)_{\partial \Omega} \quad \forall \phi \in H^1(\Omega).
	\]
	Moreover, let $z_h = R_h z$ be the corresponding discrete solution, i.e.
	\[
	z_h \in \widehat V_h \quad:\quad (\nabla \phi_h,\nabla z_h) + (\phi_h,z_h) = (u-u_h,\phi_h)_{\partial \Omega} \quad \forall \phi_h \in \widehat V_h.
	\]
	We obtain
	\[
	\begin{aligned}
		\ltwonormdo{u-u_h}^2 &= (\nabla (u-u_h),\nabla z) + (u-u_h,z)
		=(\nabla (u-i_h u),\nabla (z-z_h)) + (u-i_h u,z-z_h),\\
	\end{aligned}
	\]
	where we used Galerkin orthogonality for both $u-u_h$ and $z-z_h$. We insert positive and negative powers of $\tilde \rho$, use the Cauchy-Schwarz inequality and the boundedness of $\tilde \rho$ in $\Omega$ such that
	\begin{multline}\label{FEM:eq:in_proof_Neumanna_imr_est_start}
		\ltwonormdo{u-u_h}^2 \le c\left(\ltwonorm{\tilde \rho^{-\half}\nabla (u-i_h u)}+\ltwonorm{\tilde \rho^{-\half}(u-i_h u)}\right) \\ \times \left( \ltwonorm{\tilde \rho^{\half}\nabla (z-z_h)}+\ltwonorm{\tilde\rho^{-\half}(z-z_h)}\right).
	\end{multline}
	To estimate the weighted interpolation error we use \cref{FEM:lemma:est_interpolation_general_weight} leading to
	\[
	\begin{aligned}
		&\ltwonorm{\tilde \rho^{-\half}\nabla (u-i_h u)} + \ltwonorm{\tilde \rho^{-\half}(u-i_h u)} \le c h \ltwonorm{\tilde \rho^{-\half} \nabla^2 u} \\
		&\quad= c h \ltwonorm{\tilde \rho^{s-\half}\tilde \rho^{-s} \nabla^2 u}\le c h h^{s-\half} \ltwonorm{\tilde \rho^{-s} \nabla^2 u} \le c h^{\half + s} \ltwonorm{\rho^{-s} \nabla^2 u},
	\end{aligned}
	\]
	where we used $\tilde \rho \ge h$ and therefore $\tilde \rho^{s-\half} \le c h^{s-\half}$ (since $s<\frac12$) as well as $\tilde \rho \ge \rho$ and therefore $\tilde \rho^{-s}\le \rho^{-s}$ (since $s>0$).
	The weighted error of $z$ is estimated using \cref{FEM:Neumann:weighted_best_approximation} with the Scott-Zhang interpolation $v_h = \tilde i_h z$, see \cite{ScottZhang:1990}, as
	\begin{multline}\label{FEM:eq:in_proof_impr_est_Neumann:weighted_z}
		\ltwonorm{\tilde \rho^{-\half}(z-z_h)}+\ltwonorm{\tilde \rho^{\half}\nabla (z-z_h)}\\ \le c \ltwonorm{\tilde \rho^{\half}\nabla (z-\tilde i_h z)} + c\ltwonorm{\tilde \rho^{-\half}(z-\tilde i_h z)}.
	\end{multline}
	To estimate the above weighted interpolation errors we introduce subdomains (strips)
	\[
	B_h = \Set{x \in \Omega | \rho(x) \le h}\quad\text{and}\quad B_{bh} = \Set{x \in \Omega | \rho(x) \le b h},
	\]
	where $b>1$ is chosen large enough, such that for every cell $K$ with $K \cap (\Omega \setminus B_{bh}) \neq \emptyset$ with the corresponding patch $\omega_K$, there holds
	\[
	\dist(\omega_K,\partial \Omega) \ge h.
	\]
	We obtain
	\[
	\ltwonorm{\tilde \rho^{\half}\nabla (z-\tilde i_h z)}^2 = \norm{\tilde \rho^{\half}\nabla (z-\tilde i_h z)}_{L^2(B_{bh})}^2 +\norm{\tilde \rho^{\half}\nabla (z-\tilde i_h z)}_{L^2(\Omega\setminus B_{bh})}^2.
	\]
	For the error on the strip $B_{bh}$ we use $\tilde \rho \le c h$ and obtain
	\[
	\begin{aligned}
		\norm{\tilde \rho^{\half}\nabla (z-\tilde i_h z)}_{L^2(B_{bh})}
		&\le c h^\half \norm{\nabla (z-\tilde i_h z)}_{L^2(B_{bh})}
		 \le c h^\half \ltwonorm{\nabla (z-\tilde i_h z)}\\
		&\le c h \norm{z}_{H^{\nicefrac{3}{2}}(\Omega)}
		 \le c h \ltwonormdo{u-u_h},
	\end{aligned}
	\]
	where we used a fractional order interpolation error estimate and the $H^{\nicefrac{3}{2}}(\Omega)$ regularity for $z$ from \cref{theorem:regularity_state_adjoint}. For the error on the complement of $B_{bh}$ we can use weighted  interpolation estimates for the Scott-Zhang interpolation, similar to \cref{FEM:lemma:est_interpolation_general_weight}, and obtain
	\[
	\begin{aligned}
		\norm{\tilde \rho^{\half}\nabla (z-\tilde i_h z)}_{L^2(\Omega\setminus B_{bh})}
		&\le c h \norm{\tilde \rho^{\half} \nabla^2 z}_{L^2(\Omega \setminus B_h)}
		 \le c h \norm{\rho^{\half} \nabla^2 z}_{L^2(\Omega \setminus B_h)}\\
		&\le c h \ltwonorm{\rho^{\half} \nabla^2 z}
		 \le c h \ltwonormdo{u-u_h},
	\end{aligned}
	\]
	where we used $\rho(x) \ge ch$ on $\Omega \setminus B_h$ and therefore $\tilde \rho(x) \le c\rho(x)$. Moreover, we used the weighted regularity from \cref{theorem:regularity_state_adjoint}. The second term on the right-hand side of \cref{FEM:eq:in_proof_impr_est_Neumann:weighted_z} is simply estimated using  $\tilde \rho \ge c h$ as
	\[
	\begin{aligned}
		\ltwonorm{\tilde \rho^{-\half}(z-\tilde i_h z)} &\le ch^{-\half} \ltwonorm{z-\tilde i_h z}\le c h\norm{z}_{H^{\nicefrac{3}{2}}(\Omega)}\le c h \ltwonormdo{u-u_h},
	\end{aligned}
	\]
	where we again used a fractional order interpolation error estimate and the $H^{\nicefrac{3}{2}}(\Omega)$ regularity for $z$ from \cref{theorem:regularity_state_adjoint}.
	Putting terms together, inserting the weighted interpolation error estimates for $z$ into \cref{FEM:eq:in_proof_impr_est_Neumann:weighted_z} and then the result and the weighted interpolation error estimates for $u$ into \cref{FEM:eq:in_proof_Neumanna_imr_est_start} yields the assertion after dividing by $\ltwonormdo{u-u_h}$.
\end{proof}

\begin{corollary}\label{FEM:Corollary:u-uhboundary1}
	Let $0<s<s_\Omega$, where $s_\Omega$ is defined by \eqref{eq:s_Omega}. Let $g=0$, $f \in H^1(\Omega)$ and $u$ be the solution of \eqref{eq:z_rhs_f}. Let $u_h = R_h u$ be its Neumann Ritz projection. Then there holds
	\[
	\ltwonormdo{u-u_h} \le \frac{c}{1-2s} h^{\frac{3}{2}+s} \norm{f}_{H^1(\Omega)}
	\]
	with a positive constant $c$ only depending on $\Omega$.
\end{corollary}
\begin{proof}
	This is a direct consequence of \cref{FEM:Theorem:u-uhboundary} and \cref{theorem:weighted_H2reg}.
\end{proof}

\begin{corollary}\label{FEM:Corollary:u-uhboundary2}
	Let the critical exponent $\lambda_\Omega$ \cref{eq:lambda_Omega} fulfill $\lambda_\Omega >\frac{3}{2}$. Let $f \in H^1(\Omega)$, $g=0$ and $u$ be the solution of \eqref{eq:z_rhs_f}. Let $u_h = R_h u$ be its Neumann Ritz projection. Then there holds
	\[
	\ltwonormdo{u-u_h} \le c h^2 \lh \norm{f}_{H^1(\Omega)}
	\]
	with a positive constant $c$ only depending on $\Omega$.
\end{corollary}
\begin{proof}
 	If $\lambda_\Omega >\frac{3}{2}$, there holds $s_\Omega=\frac12$ such that we can choose $s=\frac12-\abs{\ln h}^{-1}$ in \cref{FEM:Corollary:u-uhboundary1} which yields the assertion by observing that
 	\[
 		1-2s=2\abs{\ln h}^{-1}\quad\text{and}\quad h^{-\abs{\ln h}^{-1}}=h^{(\ln h)^{-1}}=\mathrm{e}^{\ln h (\ln h)^{-1}}=\mathrm{e}.
 	\]
\end{proof}

\section{Discretization of the optimal control problem.}\label{Sec:5}
To discretize the optimal control problem \eqref{NeumannCon:eq:reduced_problem} we introduce the discrete solution operator $S_h \colon Q \to \widehat V_h$ mapping a given control $q \in Q$ to the solution $u_h = u_h(q)$ of \cref{FEM:eq:poisson_inhomogenNeumann} with $g = q$ and $f= 0$. Then, the discrete reduced cost functional is given as
\[
j_h \colon Q \to \R, \quad j_h(q) = J(q,S_h q).
\]
We discuss different discretization concepts for the control: cell-wise constant discretization, and cell-wise linear but globally continuous discretization. We also note that in case of no control constraints the latter discretization approach coincides with the concept of variational discretization.
In case of the cell-wise constant discretization we choose
\[
	Q_h=Q_h^0:=\Set{q\in Q | q|_F\in \mathcal{P}_0(F),\,\forall F\in\T_h^\partial}
\]
with $\mathcal{P}_0(F)$ denoting the spaces of constant functions on $F$. For the cell-wise linear but globally continuous discretization we set $Q_h=V_h^\partial$, see \cref{FEM:eq:V_h^partial} for the definition of $V_h^\partial$.
For any $q\in Q$, we also define the $L^2(\partial\Omega)$-projection $\pi_h\colon L^2(\partial\Omega)\to Q_h$ by
\[
	(\pi_h q - q,q_h)_{\partial\Omega}=0\quad \text{for all }q_h\in Q_h.
\]
Depending on the specific choice of the space $Q_h$, the considerations from above lead to the discrete problem:
\begin{equation}\label{discreteNeumannproblem}
	\text{Minimize } j_h(q_h),\quad q_h\in Q_h.
\end{equation}
Completely analogously to the continuous problem one can show that the functional $j_h$ is continuous and strictly convex. As a consequence, the existence and uniqueness of a discrete optimal control $\bar q_h\in Q_h$ follows by standard arguments for each of the discretization concepts. 
The functional $j_h$ is two times Fr\'echet differentiable. For a control $q \in Q$ and a direction $\dq \in Q$ the directional derivatives are
given by
\[
j_h'(q)(\dq) = (u_h-u_d,\du_h) + \alpha(q,\dq)_{\partial\Om}
\]
and
\[
	j_h''(q)(\dq,\dq) = \ltwonorm{\du_h}^2 + \alpha \ltwonormdo{\dq}^2,
\]
where $u_h = S_hq$ and $\du_h = S_h\dq$. In order to get an adjoint representation of $j_h'(q)(\dq)$ we introduce a discrete adjoint state $z_h$. As before, let $q\in Q$ and $u_h=S_hq$. We define the discrete adjoint state $z_h$ by
\begin{equation}\label{eq:discreteadjoint}
	z_h \in \widehat V_h \quad:\quad (\nabla \phi_h,\nabla z_h) + (\phi_h,z_h) = (u_h-u_d,\phi_h) \quad \text{for all } \phi_h \in \widehat V_h.
\end{equation}
As a consequence, we can reformulate the directional derivative of $j_h$ by
\[
	j_h'(q)(\dq)=(\alpha q + z_h|_{\partial \Omega},\dq)_{\partial \Omega}.
\]
The optimal control $\bar q_h\in Q_h$ is characterized by 
\begin{equation}\label{eq:vargradienteq}
	\bar q_h\in Q_h\quad:\quad j_h'(\bar q_h)(\delta q_h)=0\quad \text{for all }\delta q_h\in Q_h.
\end{equation}
Let us denote by $\bar z_h\in \widehat V_h$ the solution to \cref{eq:discreteadjoint} with right hand side $\bar u_h-u_d$, where $\bar u_h=S_h\bar q_h$ is the optimal discrete sate.
As a direct consequence of \cref{eq:vargradienteq}, we get the gradient equation
\begin{equation}\label{eq:gradienteq}
	\alpha \bar q_h + \pi_h\bar z_h|_{\partial\Omega}=0.
\end{equation}
If $Q_h=V_h^\partial$, there further holds $\pi_h \bar z_h|_{\partial\Omega} =\bar z_h|_{\partial\Omega}$ due to the definition of $\pi_h$ and the fact that $\bar z_h|_{\partial\Omega}\in V_h^\partial$.

When deriving corresponding discretization error estimates, we need several estimates for the directional derivatives.
\begin{lemma}\label{lemma:jhl2projection}
	Let $\bar q_h\in Q_h$ be the solution of \cref{discreteNeumannproblem}. If $Q_h=Q_h^0$, then for all $\delta q\in H^1(\partial\Omega)$ there holds
	\[
	\abs{j_h'(\bar q_h)(\pi_h\delta q-\delta q)}\le
		c h^2 \norm{u_d}_{L^2(\Omega)}\norm{\delta q}_{H^{1}(\partial\Omega)}
	\]
	with a constant $c>0$ only depending on $\Omega$.
	If $Q_h=V_h^\partial$, then for all $\delta q\in L^2(\partial\Omega)$ there holds
	\[
	j_h'(\bar q_h)(\delta q)=0.
	\]
\end{lemma}
\begin{proof}
	We start with considering the case $Q_h=Q_h^0$. We first use the definition of $\pi_h$ and the fact that $\bar q_h\in Q_h$. This yields
	\begin{align*}
		\abs{j_h'(\bar q_h)(\pi_h\delta q-\delta q)}&=\abs{ (S_h\bar q_h-u_d,S_h(\pi_h \delta q- \delta q))+\alpha(\bar q_h,\pi_h\delta q- \delta q)_{\partial\Omega}}\\
		&=\abs{(S_h\bar q_h-u_d,S_h(\pi_h \delta q- \delta q))}\\
		&\le \norm{S_h\bar q_h-u_d}_{L^2(\Omega)}\norm{S_h(\pi_h \delta q- \delta q)}_{L^2(\Omega)}.
	\end{align*}
	Due to the optimality of $\bar q_h$, we deduce
	\[
	\norm{S_h\bar q_h-u_d}_{L^2(\Omega)}^2\le 2j_h(\bar q_h)\le 2j_h(0)=\norm{u_d}_{L^2(\Omega)}^2.
	\]
	From \cref{theorem:L2stability} and a standard estimate for $\pi_h$ we obtain
	\begin{align*}
		\norm{S_h(\pi_h \delta q- \delta q)}_{L^2(\Omega)}
		&\le c(h \norm{\delta q-\pi_h \delta q}_{H^{-\nicefrac12}(\partial\Omega)}+\norm{\delta q-\pi_h \delta q}_{H^{-1}(\partial\Omega)})\\
		&\le c h^2 \norm{\delta q}_{H^{1}(\partial\Omega)}.
	\end{align*}
	Collecting the different estimates shows the assertion for $Q_h=Q_h^0$.
	Next, we consider the case $Q_h=V_h^\partial$. As $\alpha \bar q_h+\bar z_h|_{\partial\Omega}=0$, we get for all $\delta q\in L^2(\partial\Omega)$
	\[
		j_h'(\bar q_h)(\delta q)=0.
	\]
\end{proof}

We next show estimates for the finite element error of the directional derivatives under minimal regularity assumptions for the desired state $u_d$, i.e., $u_d\in L^2(\Omega)$. Afterwards, we show that the first estimate can be improved for $u_d\in H^1(\Omega)$.

\begin{lemma}\label{lemma:fem_derivatives2}
	Let $u_d\in L^2(\Omega)$ and let $\bar q$ solve \cref{NeumannCon:eq:reduced_problem}.
	There holds
	\[
	\abs{j'(\bar q)(\delta q)-j_h'(\bar q)(\delta q)}\le \begin{cases} c h^{\frac32}\norm{u_d}_{L^2(\Omega)} \norm{\delta q}_{L^2(\partial\Omega)}&\text{for any } \delta q\in L^2(\partial\Omega),\\
		ch^{2}\norm{u_d}_{L^2(\Omega)} \norm{\delta q}_{H^{\nicefrac{1}{2}}(\partial\Omega)}&\text{for any }\delta q\in H^{\nicefrac{1}{2}}(\partial\Omega)
	\end{cases}
	\]
	with a constant $c>0$ only depending on $\Omega$.
\end{lemma}
\begin{proof}
	Let us define $\tilde z_h \in \widehat V_h$ as the solution of \eqref{eq:discreteadjoint} with right hand side $\tilde u_h-u_d$ where $\tilde u_h=S_h\bar q$. Due to the definition of $j'$ and $j_h'$ we get
	\begin{align*}
		j'(\bar q)(\delta q)-j_h'(\bar q)(\delta q)&=(\bar z - \tilde z_h,\delta q)_{\partial\Omega}\\
		&\le \begin{cases}
			\norm{\bar z - \tilde 	z_h}_{L^2(\partial\Omega)}\norm{\delta q}_{L^2(\partial\Omega)}& \text{if } \delta q\in L^2(\partial\Omega),\\
			\norm{\bar z - \tilde 	z_h}_{H^{-\nicefrac{1}{2}}(\partial\Omega)}\norm{\delta q}_{H^{\nicefrac{1}{2}}(\partial\Omega)} &\text{if }\delta q\in H^{\nicefrac{1}{2}}(\partial\Omega).
		\end{cases}
	\end{align*}
	Additionally, let us define $\hat z_h \in \widehat V_h$ as the solution of \eqref{eq:discreteadjoint} with right hand side $\bar u-u_d$. Introducing $\hat z_h$ as an intermediate function yields
	\begin{align*}
		\norm{\bar z - \tilde z_h}_{L^2(\partial\Omega)}&\le \norm{\bar z - \hat z_h}_{L^2(\partial\Omega)}+\norm{\hat z_h - \tilde 	z_h}_{L^2(\partial\Omega)}
		\intertext{and}
		\norm{\bar z - \tilde z_h}_{H^{-\nicefrac{1}{2}}(\partial\Omega)}&\le \norm{\bar z - \hat z_h}_{H^{-\nicefrac{1}{2}}(\partial\Omega)}+\norm{\hat z_h - \tilde 	z_h}_{H^{-\nicefrac{1}{2}}(\partial\Omega)},
	\end{align*}
	respectively.
	Each of the first terms on the right hand sides can estimated by \cref{FEM:lemma:u-uhboundary_negative} in combination with the regularity results of \cref{NeumannCon:theorem:opt_sys}. By this we obtain
	\[
	\norm{\bar z - \hat z_h}_{H^{-\nicefrac{1}{2}}(\partial\Omega)}+h^{\frac12}\norm{\bar z - \hat z_h}_{L^2(\partial\Omega)}\le c h^{2}\norm{u_d}_{L^2(\Omega)}.
	\]
	Let us consider the second terms. Using the embedding $L^2(\partial\Omega)\hookrightarrow H^{-\nicefrac{1}{2}}(\partial\Omega)$, a standard trace theorem and a standard a priori estimate, we get 
	\[
	\norm{\hat z_h - \tilde z_h}_{H^{-\nicefrac{1}{2}}(\partial\Omega)}\le c\norm{\hat z_h - \tilde z_h}_{L^2(\partial\Omega)}\le c \norm{\hat z_h - \tilde z_h}_{H^1(\Omega)}\le c \norm{\bar u - \tilde u_h}_{L^2(\Omega)}.
	\]
	The latter term represents a finite element error in the domain. Using \cref{theorem:standard_error_est_H2} and the a priori estimate from \cref{NeumannCon:theorem:opt_sys}, we obtain
	\[
		\norm{\bar u - \tilde u_h}_{L^2(\Omega)}\le c h^2 \norm{u_d}_{L^2(\Omega)}.
	\]
	Hence, we get
	\begin{equation}\label{eq:inter1}
		\norm{\hat z_h - \tilde z_h}_{H^{-\nicefrac{1}{2}}(\partial\Omega)} + \norm{\hat z_h - \tilde z_h}_{L^2(\partial\Omega)}\le c h^2 \norm{u_d}_{L^2(\Omega)}.
	\end{equation}
	Collecting the different results yields the assertion.
\end{proof}

\begin{lemma}\label{lemma:fem_derivatives}
	Let $u_d\in H^1(\Omega)$ and let $\bar q$ solve \cref{NeumannCon:eq:reduced_problem}. For $0<s<s_\Omega$ with $s_\Omega$ from \cref{eq:s_Omega} there holds for all $\delta q\in L^2(\partial\Omega)$
	\[
		\abs{j'(\bar q)(\delta q)-j_h'(\bar q)(\delta q)}\le c h^{\frac32+s}\norm{u_d}_{H^1(\Omega)} \norm{\delta q}_{L^2(\partial\Omega)}
	\]
	with a constant $c>0$ only depending on $\Omega$.
	For $\lambda_\Omega>\frac32$ with $\lambda_\Omega$ from \cref{eq:lambda_Omega} there even holds
	\[
		\abs{j'(\bar q)(\delta q)-j_h'(\bar q)(\delta q)}\le c h^{2}\abs{\ln h}\norm{u_d}_{H^1(\Omega)} \norm{\delta q}_{L^2(\partial\Omega)}
	\]
	with a constant $c>0$ only depending on $\Omega$.
\end{lemma}
\begin{proof}
	As in the proof of the previous lemma, we obtain
	\[
		\abs{j'(\bar q)(\delta q)-j_h'(\bar q)(\delta q)}\le \left(\norm{\bar z - \hat z_h}_{L^2(\partial\Omega)}+\norm{\hat z_h - \tilde z_h}_{L^2(\partial\Omega)}\right)\norm{\delta q}_{L^2(\partial\Omega)}.
	\]
	The second term of the sum was estimated in \cref{eq:inter1} yielding
	\[
		\norm{\hat z_h - \tilde z_h}_{L^2(\partial\Omega)}\le c h^2 \norm{u_d}_{L^2(\Omega)}\le c h^2 \norm{u_d}_{H^1(\Omega)}.
	\]
	The first term was treated in \cref{FEM:Corollary:u-uhboundary1} and \cref{FEM:Corollary:u-uhboundary2} for the case that $u_d$ belongs to $H^1(\Omega)$. In summary, we get for $0<s<s_\Omega$
	\[
	\norm{\bar z - \hat z_h}_{L^2(\partial\Omega)}\le ch^{\frac32+s}\norm{\bar u-u_d}_{H^1(\Omega)}\le c h^{\frac32+s}\norm{u_d}_{H^1(\Omega)},
	\]
	if we apply \cref{FEM:Corollary:u-uhboundary1}. Note, that last step holds due to \cref{NeumannCon:theorem:opt_sys}. If we use \cref{FEM:Corollary:u-uhboundary2} instead, we obtain assuming $\lambda_\Omega>\frac32$
	\[
	\norm{\bar z - \hat z_h}_{L^2(\partial\Omega)}\le  c h^{2}\abs{\ln h}\norm{u_d}_{H^1(\Omega)}.
	\]
	Putting terms together completes the proof.
\end{proof}

Next we prove optimal convergence rates for the controls. We distinguish between more and less regular desired states.
\begin{theorem}\label{theorem:main1}
	Let $\bar q$ and $\bar q_h$ be the solution of \cref{NeumannCon:eq:reduced_problem} and \cref{discreteNeumannproblem}, respectively.
	\begin{enumerate}
		\item For $u_d\in L^2(\Omega)$ there holds
			\[
				\norm{\bar q-\bar q_h}_{L^2(\partial\Omega)}\le
				\begin{cases}
					ch\norm{u_d}_{L^2(\Omega)} &\text{if } Q_h=Q_h^0,\\
					ch^{\frac32}\norm{u_d}_{L^2(\Omega)} &\text{if } Q_h=V_h^\partial,
				\end{cases}
			\]
			with a constant $c>0$ only depending on $\Omega$.
		\item Let $u_d\in H^1(\Omega)$ and $Q_h=V_h^\partial$. There holds with a constant $c>0$ only depending on $\Omega$ that
		\[
		\norm{\bar q-\bar q_h}_{L^2(\partial\Omega)}\le
		\begin{cases}
			c h^{\frac32+s}\norm{u_d}_{H^1(\Omega)}&\quad\text{for } 0<s<s_\Omega,\\
			c h^{2}\abs{\ln h}\norm{u_d}_{H^1(\Omega)}&\quad\text{for }\lambda_\Omega >\frac32,
		\end{cases}
		\]
		with $s_\Omega$ from \cref{eq:s_Omega} and $\lambda_\Omega$ from \cref{eq:lambda_Omega}.
	\end{enumerate}
\end{theorem}
\begin{proof}
	Due to the coercivity of the second order directional derivatives and the linear quadratic structure of the optimal control problems we obtain
	\begin{align*}
		\alpha \norm{\bar q-\bar q_h}_{L^2(\partial\Omega)}^2&\le j_h''(q)(\bar q-\bar q_h,\bar q-\bar q_h)=j_h'(\bar q)(\bar q-\bar q_h)-j_h'(\bar q_h)(\bar q-\bar q_h)=I+II
	\end{align*}
	where $q\in L^2(\partial\Omega)$ is arbitrary.
	The first term can be estimated as follows: Due to the optimality of $\bar q$ we get
	\[
		I=j_h'(\bar q)(\bar q-\bar q_h)-j'(\bar q)(\bar q-\bar q_h).
	\]
	This term was estimated in \cref{lemma:fem_derivatives2} and \cref{lemma:fem_derivatives}. To be more precise, for $u_d\in L^2(\Omega)$ we obtain
	\[
		I\le c h^{\frac32}\norm{u_d}_{L^2(\Omega)}\norm{\bar q-\bar q_h}_{L^2(\partial\Omega)}
	\]
	independent of the discretization strategy.
	If $u_d$ belongs to $H^1(\Omega)$ we get
	\[
		I\le 
		\begin{cases}
			c h^{\frac32+s}\norm{u_d}_{H^1(\Omega)}\norm{\bar q-\bar q_h}_{L^2(\partial\Omega)}&\quad\text{for } 0<s<s_\Omega,\\
			c h^{2}\abs{\ln h}\norm{u_d}_{H^1(\Omega)}\norm{\bar q-\bar q_h}_{L^2(\partial\Omega)}&\quad\text{for }\lambda_\Omega >\frac32.
		\end{cases}
	\]
	Let us now consider the second term:
	By the optimality of $\bar q_h$, we obtain
	\[
	II=j_h'(\bar q_h)(\bar q_h-\bar q)=j_h'(\bar q_h)(\bar q_h - \pi_h\bar q)+j_h'(\bar q_h)(\pi_h\bar q- \bar q)=j_h'(\bar q_h)(\pi_h\bar q - \bar q).
	\]
	Next, we use \cref{lemma:jhl2projection} in combination with the regularity results from \cref{NeumannCon:theorem:opt_sys}. This yields
	\[
	j_h'(\bar q_h)(\pi_h\bar q - \bar q)\le \begin{cases}
		c h^2 \norm{u_d}_{L^2(\Omega)}^2 &\text{if }Q_h=Q_h^0,\\
		0 & \text{if }Q_h=V_h^\partial.
	\end{cases}
	\]
	Finally, the assertion follows from the estimates for $I$ and $II$ using Young's inequality.
\end{proof}
The following theorem is concerned with optimal convergence rates for the states. The proof employs a duality argument and relies on the fact that no control bounds are present. A similar argument was used in \cite{MayRannacherVexler:2013} to show higher rates for the states in case of Dirichlet boundary control problems. However, there the duality argument was accomplished on the whole optimality system, whereas in the present case, we work with the gradient equations only.
\begin{theorem}\label{theorem:main2}
	Let $\Omega \subset \R^3$ be an arbitrary convex polyhedral domain and let $u_d\in L^2(\Omega)$. Let $\bar q$ be the solution of \cref{NeumannCon:eq:reduced_problem} with corresponding state $\bar u=S\bar q$. Moreover, let $\bar q_h$ be the solution of \cref{discreteNeumannproblem} with $Q_h=Q_h^0$ or $Q_h=V_h^\partial$, respectively, and let $\bar u_h=S_h\bar q_h$ be the corresponding discrete state. Then there holds the estimate
	\[
		\norm{\bar u - \bar u_h}_{L^2(\Omega)}\le c h^2 \norm{u_d}_{L^2(\Omega)}
	\]
	with a constant $c>0$ only depending on $\Omega$.
\end{theorem}
\begin{proof}
We define the functionals $l:Q\to\R$ and $l_h:Q\to\R$ by 
\[
	l(w)= \frac12\norm{Sw-w_d}_{L^2(\Omega)}^2+\frac\alpha2\norm{w}_{L^2(\partial\Omega)}^2
\]
and
\[
	l_h(w)= \frac12\norm{S_hw-w_d}_{L^2(\Omega)}^2+\frac\alpha2\norm{w}_{L^2(\partial\Omega)}^2,
\]
respectively, with $w_d=\bar u - \bar u_h=S\bar q-S_h\bar q_h$. Next, let $\bar w\in Q$ be the unique solution to the optimization problem
\[
\min_{w\in Q}l(w)
\]
and let $\bar w_h\in Q_h$ be the unique solution to the corresponding discrete optimization problem
\[
\min_{w_h\in Q_h}l_h(w_h).
\]
Obviously, $l$ and $l_h$ coincide with $j$ and $j_h$, respectively, if $u_d=w_d$. Consequently, all the estimates for $j'$ and $j_h'$ are also valid for $l'$ and $l_h'$. Moreover, estimates from \cref{theorem:main1} are also applicable for $\bar w$ and $\bar w_h$. We now apply the duality argument to derive the estimate from the assertion.
Direct calculations yield
\begin{align*}
	\norm{S\bar q-S_h\bar q_h}_{L^2(\Omega)}^2
	=(w_d,S\bar q-S_h\bar q_h)=(w_d,(S-S_h)\bar q)+(w_d,S_h(\bar q-\bar q_h)).
\end{align*}
The first term can be estimated by \cref{theorem:standard_error_est_H2} and \cref{NeumannCon:theorem:opt_sys} resulting in
\[
	(w_d,(S-S_h)\bar q)\le \norm{w_d}_{L^2(\Omega)}\norm{(S-S_h)\bar q}_{L^2(\Omega)}\le c h^2  \norm{w_d}_{L^2(\Omega)} \norm{u_d}_{L^2(\Omega)}.
\]
Let us consider the second term. We introduce another intermediate term, i.e.,
\[
	(w_d,S_h(\bar q-\bar q_h))=(w_d,S_h(\bar q-\pi_h\bar q))+(w_d,S_h(\pi_h\bar q-\bar q_h)).
\]
Due to the optimality of $\bar w_h$ and $\bar q_h$, we get
\begin{align*}
	(w_d,S_h(\pi_h\bar q-\bar q_h))&=(S_h\bar w_h,S_h(\pi_h\bar q-\bar q_h)) + \alpha(\bar w_h,\pi_h\bar q-\bar q_h)_{\partial\Omega}\\
	&=j_h'(\pi_h\bar q)(\bar w_h)-j_h'(\bar q_h)(\bar w_h)\\
	&=j_h'(\pi_h\bar q)(\bar w_h).
\end{align*}
The definition of $j_h'$ and $l_h'$ implies
\begin{align*}
	j_h'(\pi_h\bar q)(\bar w_h)&=j_h'(\pi_h\bar q)(\bar w_h)-j_h'(\bar q)(\bar w_h)+j_h'(\bar q)(\bar w_h)\\
	&=(w_d,S_h(\pi_h\bar q-\bar q))+l_h'(\bar w_h)(\pi_h \bar q-\bar q)+j_h'(\bar q)(\bar w_h).
\end{align*}
In summary, we arrive at
\[
	(w_d,S_h(\bar q-\bar q_h))=l_h'(\bar w_h)(\pi_h \bar q-\bar q)+j_h'(\bar q)(\bar w_h).
\]
Using the optimality of $\bar q$ we get
\[
	(w_d,S_h(\bar q-\bar q_h))=l_h'(\bar w_h)(\pi_h\bar q-\bar q)+j_h'(\bar q)(\bar w_h-\bar w)-j'(\bar q)(\bar w_h-\bar w)+j_h'(\bar q)(\bar w)-j'(\bar q)(\bar w).
\]
For the first term, we obtain by means of \cref{lemma:jhl2projection} and the regularity results from \cref{NeumannCon:theorem:opt_sys}
\[
	\abs{l_h'(\bar w_h)(\pi_h\bar q-\bar q)}\le \begin{cases}
		c h^2 \norm{w_d}_{L^2(\Omega)}\norm{u_d}_{L^2(\Omega)} &\text{if }Q_h=Q_h^0,\\
		0 & \text{if }Q_h=V_h^\partial.
	\end{cases}
\]
The other terms were estimated in \cref{lemma:fem_derivatives2} yielding
\[
	j_h'(\bar q)(\bar w_h-\bar w)-j'(\bar q)(\bar w_h-\bar w)\le ch^{\frac32}\norm{\bar w_h-\bar w}_{L^2(\partial\Omega)}\norm{u_d}_{L^2(\Omega)}
\]
and
\[
	j_h'(\bar q)(\bar w)-j'(\bar q)(\bar w)\le ch^2\norm{\bar w}_{H^{\nicefrac{1}{2}}(\partial\Omega)}\norm{u_d}_{L^2(\Omega)}.
\]
Using the estimate for $\bar w-\bar w_h$ from \cref{theorem:main1} for $w_d\in L^2(\Omega)$ and the regularity estimate from \cref{NeumannCon:theorem:opt_sys} results in
\[
	j_h'(\bar q)(\bar w_h-\bar w)-j'(\bar q)(\bar w_h-\bar w)+j_h'(\bar q)(\bar w)-j'(\bar q)(\bar w)\le c h^2\norm{w_d}_{L^2(\Omega)}\norm{u_d}_{L^2(\Omega)}.
\]
Collecting the different results and dividing by $\norm{w_d}_{L^2(\Omega)}$ yield the assertion.
\end{proof}

\section{Numerical results}\label{Sec:6}
In this section, we present a numerical example confirming the convergence rates which we have proven in the previous section.
To be able to state an example with a known solution, we consider the optimal control problem \cref{NeumannCon:eq:problem} with a slightly modified state equation, i.e.,  we consider
\[
\begin{aligned}
	-\Lap u + u &= f &\quad&\text{in } \Omega_\omega,\\
	\partial_n u &=q+g &\quad&\text{on } \partial \Omega_\omega
\end{aligned}
\]
in the domain $\Om_\omega\subset\R^3$ given as
\[
\Om_\omega=\left((-1,1)^2\cap \Set{(r\cos\varphi,r\sin\varphi)^T |
	r\in(0,\infty),~\varphi\in(0,\omega)}\right)\times(0,1)
\]
for $\omega=\omega_\Om \in[\frac\pi2,\pi)$, see \cref{NeumannCon:fig:Omega3D}. Here, we use cylindrical coordinates $(r,\varphi,x_3)$ of
$x=(x_1,x_2,x_3)^T\in\R^3$, where $r$ and $\varphi$ denote the polar coordinates of $(x_1,x_2)^T\in\R^2$ given by
\[
r(x)=\sqrt{x_1^2+x_2^2}\quad\text{and}\quad \varphi(x)=\operatorname{arctan2}(x_2,x_1).
\] 
The optimal state and adjoint state are chosen as
\[
\ou(x)=\oz(x)=r(x)^\lambda \cos(\lambda\phi(x))\eta(r(x))\cos(\pi x_3),
\]
with $\lambda=\frac\pi\omega$ and $\eta$ being a cut-off function given by
\[
\eta(r)=\begin{cases}
	(1+3r)(1-r)^3,&\text{for }0\le r\le 1,\\
	0,&\text{for }r>1.
\end{cases}
\]

\tdplotsetmaincoords{70}{-10}

\begin{figure}
	\centering
	\begin{tikzpicture}[scale=1.9, tdplot_main_coords]
		\draw (0,0,1) -- (1,0,1) -- (1,1,1) -- (0,1,1) -- cycle;
		\draw (0,0,0) -- (1,0,0);
		\draw[dashed] (1,0,0) -- (1,1,0) -- (0,1,0);
		\draw (0,1,0) -- (0,0,0);
		\draw (0,0,0) -- (0,0,1);
		\draw (1,0,0) -- (1,0,1);
		\draw[dashed] (1,1,0) -- (1,1,1);
		\draw (0,1,0) -- (0,1,1);
		\node at (0.5,0.5,0.5) {$\Om_{\frac\pi2}$};
	\end{tikzpicture}\qquad
	\begin{tikzpicture}[scale=1.9, tdplot_main_coords]
		\draw (0,0,1) -- (1,0,1) -- (1,1,1) -- (-0.57735026919,1,1) -- cycle;
		\draw (0,0,0) -- (1,0,0);
		\draw[dashed] (1,0,0) -- (1,1,0) -- (-0.57735026919,1,0);
		\draw (-0.57735026919,1,0) -- (0,0,0);
		\draw (0,0,0) -- (0,0,1);
		\draw (1,0,0) -- (1,0,1);
		\draw[dashed] (1,1,0) -- (1,1,1);
		\draw (-0.57735026919,1,0) -- (-0.57735026919,1,1);
		\node at (0.3558,0.5,0.5) {$\Om_{\frac{2\pi}3}$};
	\end{tikzpicture}\qquad
	\begin{tikzpicture}[scale=1.9, tdplot_main_coords]
		\draw (0,0,1) -- (1,0,1) -- (1,1,1) -- (-1,1,1) -- cycle;
		\draw (0,0,0) -- (1,0,0);
		\draw[dashed] (1,0,0) -- (1,1,0) -- (-1,1,0);
		\draw (-1,1,0) -- (0,0,0);
		\draw (0,0,0) -- (0,0,1);
		\draw (1,0,0) -- (1,0,1);
		\draw[dashed] (1,1,0) -- (1,1,1);
		\draw (-1,1,0) -- (-1,1,1);
		\node at (0.35,0.5,0.5) {$\Om_{\frac{3\pi}4}$};
	\end{tikzpicture}
	\caption{Domain $\Om_\omega$ for
		$\omega\in\set{\frac\pi2,\frac{2\pi}3,\frac{3\pi}4}$ and $N=3$}\label{NeumannCon:fig:Omega3D}
\end{figure}
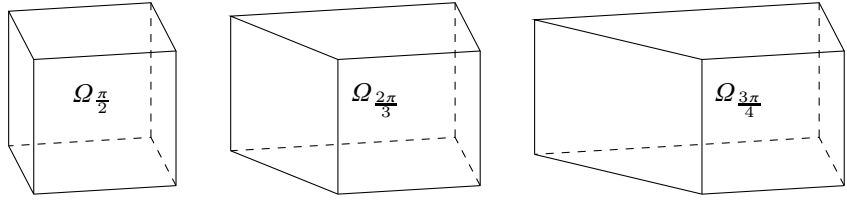

We can directly verify that $\ou$ and $\oz$ fulfill homogeneous Neumann boundary conditions on $\partial \Om_\omega$. Moreover, we set
\[
	\oq=-\oz\bigr\rvert_{\partial\Om_\omega},\quad f=-\Lap\ou+\ou,\quad g=-\oq, \quad u_d=\ou+\Lap\oz-\oz.
\]
Direct calculations show that those functions fulfill the accordingly modified optimality condition of \cref{NeumannCon:theorem:opt_sys} for $\alpha=1$.

In \cref{NeumannCon:fig:qerror3D}, we present the development of the error $\ltwonormdo{\oq-\oq_h}$
for $\omega\in\set{\frac\pi2,\frac{2\pi}3,\frac{3\pi}4}$ and $h$ tending to zero. We exactly observe the orders of convergence predicted by
\cref{theorem:main1}. In case of $Q_h=Q_h^0$, we get a convergence rate of $1$ regardless of the value
of $\omega$. For $Q_h=V_h^\partial$, we obtain a convergence order of $\frac{11}{6}$, if
$\omega=\frac{3\pi}{4}$, and convergence order of $2$ if $\omega=\frac{2\pi}{3}$ or $\omega =
\frac{\pi}{2}$, respectively. In \cref{NeumannCon:fig:uerror3D}, we present the development of the error
$\ltwonormdo{\ou-\ou_h}$ for $\omega\in\set{\frac\pi2,\frac{2\pi}3,\frac{3\pi}4}$ and $h$ tending to
zero. As proven in \cref{theorem:main2} we observe a convergence rate of $2$ regardless of the chosen discretization strategy and regardless of the angle $\omega$.
\begin{figure}
	\begin{tikzpicture}
		\begin{loglogaxis} [%
			xlabel={mesh size $h$},
			grid=major,
			legend pos=south east,
			legend cell align=left,
			width=\textwidth,
			height=0.6\textwidth,
			cycle list name=black white,
			]
			
			\addplot table [x=h, y=eq, col sep=semicolon] {errors3D_cell_12.csv};
			\addlegendentry{$\omega=\frac\pi2$, $Q_h=Q_h^0$}
			\addplot table [x=h, y=eq, col sep=semicolon] {errors3D_cell_23.csv};
			\addlegendentry{$\omega=\frac{2\pi}3$, $Q_h=Q_h^0$}
			\addplot table [x=h, y=eq, col sep=semicolon] {errors3D_cell_34.csv};
			\addlegendentry{$\omega=\frac{3\pi}4$, $Q_h=Q_h^0$}
			
			\addplot table [x=h, y=eq, col sep=semicolon] {errors3D_node_12.csv};
			\addlegendentry{$\omega=\frac\pi2$, $Q_h=V_h^\partial$}
			\addplot table [x=h, y=eq, col sep=semicolon] {errors3D_node_23.csv};
			\addlegendentry{$\omega=\frac{2\pi}3$, $Q_h=V_h^\partial$}
			\addplot table [x=h, y=eq, col sep=semicolon] {errors3D_node_34.csv};
			\addlegendentry{$\omega=\frac{3\pi}4$, $Q_h=V_h^\partial$}
			
			\logLogSlopeTriangle{0.25}{0.15}{0.47}{1}{black}{h}
			\logLogSlopeTriangleInverse{0.23}{0.15}{0.36}{1.833333}{black}{h^{\frac{11}6}}
			\logLogSlopeTriangle{0.25}{0.15}{0.073}{2}{black}{h^2}
		\end{loglogaxis}
	\end{tikzpicture}
	\caption{Error $\ltwonormdo{\oq-\oq_h}$ for $N=3$}\label{NeumannCon:fig:qerror3D}
\end{figure}

\begin{figure}
	\begin{tikzpicture}
		\begin{loglogaxis} [%
			xlabel={mesh size $h$},
			grid=major,
			legend pos=south east,
			legend cell align=left,
			width=\textwidth,
			height=0.6\textwidth,
			cycle list name=black white,
			]
			
			\addplot table [x=h, y=eu, col sep=semicolon] {errors3D_cell_12.csv};
			\addlegendentry{$\omega=\frac\pi2$, $Q_h=Q_h^0$}
			\addplot table [x=h, y=eu, col sep=semicolon] {errors3D_cell_23.csv};
			\addlegendentry{$\omega=\frac{2\pi}3$, $Q_h=Q_h^0$}
			\addplot table [x=h, y=eu, col sep=semicolon] {errors3D_cell_34.csv};
			\addlegendentry{$\omega=\frac{3\pi}4$, $Q_h=Q_h^0$}
			
			\addplot table [x=h, y=eu, col sep=semicolon] {errors3D_node_12.csv};
			\addlegendentry{$\omega=\frac\pi2$, $Q_h=V_h^\partial$}
			\addplot table [x=h, y=eu, col sep=semicolon] {errors3D_node_23.csv};
			\addlegendentry{$\omega=\frac{2\pi}3$, $Q_h=V_h^\partial$}
			\addplot table [x=h, y=eu, col sep=semicolon] {errors3D_node_34.csv};
			\addlegendentry{$\omega=\frac{3\pi}4$, $Q_h=V_h^\partial$}
			\logLogSlopeTriangle{0.25}{0.15}{0.073}{2}{black}{h^2}
		\end{loglogaxis}
	\end{tikzpicture}
	\caption{Error $\ltwonorm{\ou-\ou_h}$ for $N=3$}\label{NeumannCon:fig:uerror3D}
\end{figure}

\section{Acknowledgment}\label{sec7} We would like to thank Dr. Dominik Meidner for scientific exchange and for helping us to prepare the numerical examples.

\backmatter

\bibliography{lit}

\end{document}